\documentclass[11pt]{amsart}
 \usepackage{amssymb,amsmath,latexsym,amsthm}

 \usepackage{times}
 \usepackage{dsfont}
 \usepackage[all]{xy}

 \usepackage[english,francais]{babel}

 \usepackage{url} 
 \usepackage{hyperref}

 \usepackage{color}

 \definecolor{greenbf}{rgb}{0, 0.7 ,0.3}
 
 \hypersetup{	
colorlinks=true, %colorise les liens 
breaklinks=true, %permet le retour ˆ la ligne dans les liens trop longs 
urlcolor= blue, %couleur des hyperliens 
linkcolor= red,	%couleur des liens internes 
citecolor= {greenbf},	%couleur des rŽfŽrences 
}

\newtheorem{theorem}{{Theorem}}[section]
\newtheorem{proposition}[theorem]{{Proposition}}%[section]
\newtheorem{isom.ext}[theorem]{{Trivial isometric extension}}%[section]
\newtheorem{corollary}[theorem]{{Corollary}}%[section]
\newtheorem{fact}[theorem]{{\sc Fact}}%[section]
\newtheorem{remark}[theorem]{{Remark}}%[section]
\newtheorem{example}[theorem]{{Example}}%[section]
\newtheorem{lemma}{{Lemma}} 

\def\C{\mathbb{C}}
\def\R{\mathbb{R}}

\def\H{\mathbb{H}}

\def\U{\sf{U}}
\def\SU{\sf{SU}}

\def\P{\mathbb{P}}

\def\GL{{\sf{GL}}}
\def\O{{\sf{O}}}
\def\SO{{\sf{SO}}}
\def\SU{{\sf{SU}}} 
\def\SL{{\sf{SL}}}

\def\End{{\sf{End}}}
\def\Mat{{\sf{Mat}} }
\def\Iso{{\sf{Iso}} }

\def\Span{{\sf{Span}} }

\def\Mink{{\sf{Mink}} }
\def\dS{{\sf{dS}} }
\def\AdS{{\sf{AdS}} }
 
\def\tr{{\sf{tr}}}

\def\det{{\sf{det}}}

\def\p{{\mathfrak{p}}}
\def\k{{\mathfrak{k}}}

\def\a{{\mathfrak{a}}}
\def\g{{\mathfrak{g}}}

\def\s{{\mathfrak{s}}}
\def\h{{\mathfrak{h}}}

\def\I{{\mathfrak{I}}}

\begin{document}

\selectlanguage{english}
%***

%***

\title[Hermite-Lorentz metrics]{On homogeneous  Hermite-Lorentz spaces}
\author [A. Ben Ahmed]{Ali Ben Ahmed}
\address{D\'epartement de Math\'ematiques, Universit\'e de Tunis ElManar\hfill\break\indent
Campus universitaire, 2092, Tunis, TUNISIA}
\email{benahmedal@gmail.com}

\author[A. Zeghib]{Abdelghani Zeghib }
\address{UMPA, CNRS, 
%Unit\'e de Math\'ematiques Pures et Appliqu\'ees
%\hfill\break\indent
\'Ecole Normale Sup\'erieure de Lyon\hfill\break\indent
46, all\'ee d'Italie
%\hfill\break\indent
69364 LYON Cedex 07, FRANCE}
\email{abdelghani.zeghib@ens-lyon.fr 
\hfill\break\indent
\url{http://www.umpa.ens-lyon.fr/~zeghib/}}
\date{\today}
\maketitle

\begin{abstract} We define naturally Hermite-Lorentz metrics on almost-complex manifolds as special case of pseudo-Riemannian metrics compatible with the almost complex structure. We study their isometry groups.

\end{abstract}
 
  \tableofcontents

\section{Introduction}

\subsubsection*{Simplest pseudo-Hermitian structures} Let us call a quadratic form $q$ on a complex space of dimension $n+1$ of {\bf  Hermite-Lorentz} type if it is  $\C$-equivalent to the standard form  
$q_0 = - |z_0 |^2 + |z_1|^2 + \ldots + | z_n|^2$ on $\C^{n+1}$. In other words, $q$ is Hermitian, and as a real form, it  has a signature $- - + \ldots +$. Here, Lorentz refers to the occurrence of exactly one negative sign (in the complex presentation). Classically, this one negative sign distinguishes, roughly,  between time and space components. (A ``complex-Lorentz'' form could perhaps be an equally informative terminology?) 

One can then define Hermite-Lorentz metrics on almost complex   manifolds. If $(M, J)$ is an almost complex manifold, then $g$ is a Hermite-Lorentz metric if $g$ is a tensor such that 
for any $x \in M$, $(T_xM, J_x, g_x)$ is a Hermite-Lorentz linear space. 

Hermite-Lorentz metrics generalize (definite)  Hermitian metrics, and they are the nearest from them, among general pseudo-Hermitian structures,   in the sense that they have the minimal (non-trivial) signature.
 Our point of view  here is  to compare Hermite-Lorentz metrics,  on one hand with (definite) Hermitian metrics in complex geometry, and with Lorentz metrics in (real) differential geometry.

\subsubsection*{$H$-structure} Let $\U(1,n) \subset \GL_{n+1}(\C)$ be the unitary  group of $q_0$. Then, a Hermite-Lorentz structure on a manifold $M$ of real dimension $2n+2$ is a reduction of the structural group of $TM$ to $\U(1, n)$. They are different from ``complex Riemannian'' metrics which are reduction to $\O(n+1, \C)$.

\subsubsection*{K\"ahler-Lorentz spaces}

As in the positive definite case,  
the   K\"ahler form $\omega$  is  defined  by $\omega (u, v) = 
g(u, Jv)$. It is a $J$-invariant 2-differential form. A {\bf K\"ahler-Lorentz} metric corresponds to the case where $J$ is integrable and $\omega $ is closed.
A { K\"ahler-Lorentz} manifold is in particular symplectic. 

Conversely,  from the symplectic point of view, a symplectic manifold $(M, \omega)$ is K\"ahler-Lorentz if $\omega$ can be calibrated with a special complex structure $J$. Let us generalize the notion of calibration 
 by  letting it to mean that $J$  satisfies  that $g(u, v) = \omega(u, J v)$ is non-degenerate, i.e. $g$ is a pseudo-Hermitian  metric. Now, the classical K\"ahler case  means that $g$ is Hermitian and in addition 
 $J$ is integrable. So K\"ahler-Lorentz means that $g$ is ``post-Hermitian'' in the sense
 that it has a  Hermite-Lorentz signature.  
 
 \subsubsection*{Holomorphic sectional curvature}

Differential geometry can be developed for general pseudo-Hermitian 
  metrics exactly as in the usual Hermitian   as well as the usual  pseudo-Hermitian  cases.
In particular, there is a Levi-Civita connection and a Riemann curvature tensor $R$. 
For a tangent vector $u$, the {\bf holomorphic sectional  curvature} $K(u)$ is the sectional curvature of the real 2-plane
$\C u$; $K(u) = \frac{g(R(u, Ju)Ju, u)}{g(u,u)^2}$
 (this requires $u$ to be non isotropic $g(u, u) \neq 0$, in order to divide by the  
 volume $u\wedge Ju$). So, $K$ is a real function on an open set of the projectivization  bundle of $TM$
 (which fibers over $M$ with fiber type $\P^n(\C)$).

In fact, $K$   determines the full  Riemann tensor in the pseudo-K\"ahler case \cite{Kob-Nom, Romero1} (but not 
 in the general pseudo-Hermitian  case). In particular, the case $K$ constant in the definite  K\"ahler case corresponds to the most central homogeneous spaces: $\C^n$, $\P^n(\C)$ and $\H^n(\C)$ (the complex hyperbolic space). K\"ahler-Lorentz spaces of constant curvature are   introduced below.

 \subsection{Examples}
 
 ${}$ 
 
  We are going to give examples of homogeneous spaces $M= G/H$, where 
 the natural (generally unique) $G$-invariant geometric structure is a K\"ahler-Lorentz metric.  
 
\subsubsection{Universal K\"ahler-Lorentz spaces of constant holomorphic curvature} If  a K\"ahler-Lorentz metric has constant holomorphic sectional curvature, then it is locally isometric 
to one of the following spaces:

\begin{enumerate}

\item
 The universal (flat Hermite-Lorentz)  complex Minkowski space $\Mink_n(\C)$ (or $\C^{1, n-1}$), that is   $\C^n$ endowed with   $q_0 = - |z_1 |^2 + |z_2|^2 + \ldots + | z_{n}|^2$.

  \medskip

\item
The complex de Sitter space $\dS_n(\C) =   \SU(1, n)/ \U(1, n-1)$ 
\footnote{Here $\U(1, n-1)$ as a subgroup of $\SU(1, n)$ stands for matrices of the form
  \[ \begin{pmatrix} \begin{matrix} \lambda A &  0\\0 & \lambda^{-1} \end{matrix}
 \end{pmatrix}, \vert\lambda \vert = 1,  A \in \SU(1, n-1).  \]  In  general $\U(1, n-1)$ designs a group isomorphic to a product $\U(1) \times \SU(1, n-1)$, where  the embedding $U(1)$ depends on the context.} It has a positive constant holomorphic sectional curvature.

\medskip

\item The complex anti de Sitter space $\AdS_n(\C) =   \SU(2, n-1) / \U(1, n-1)$. It has negative curvature.

\medskip

 \end{enumerate}
 
$\bullet$  As said above, Hermite-Lorentz metrics are generalizations of both Hermitian metrics (from the definite to the indefinite) and Lorentz metrics (from the real to the complex). 
 Let us    draw up in the following table the analogous of our previous spaces in both Hermitian and Lorentzian settings.

 \bigskip
 \begin{center}

\begin{tabular}{|l|l|l|l|}
\hline
K\"ahler-Lorentz spaces & Hermitian (positive definite)   & (Real) Lorentz   \\
of constant curvature & counterpart  & counterpart  \\ \hline
$\Mink_n(\C)$ &  $\C^n$ &  $\Mink_n(\R)$  \\ \hline
 $\dS_n(\C)= $  & $\P^n(\C) = $ &  $\dS_n(\R) =$ \\ 
 \SU(1, n)/U(1, n-1) & $ \SU(1+n)/\U(n)$ & $ \SO^0(1, n) /\SO^0(1,n-1)$  \\ \hline
 $\AdS_n(\C) = $ &  $\mathbb H^n(\C)= $  &    $ \AdS_n(\R)= $ \\
 $\SU(2,n-1)/\U(1,n-1)$ &  $\SU(1,n)/\U(n)$  & $  \SO(2, n-1)/\SO^0(1,n-1)$ \\ 
\hline
\end{tabular}

 \end{center} 
 
 \bigskip
 
 (Of course, we also have as Riemannian counterparts of constant sectional curvature, respectively, the Euclidean, spherical and hyperbolic spaces, $\R^n$, $\mathbb S^n$ and $\H^n$).

$\bullet$ The K\"ahler-Lorentz spaces of constant holomorphic sectional curvature are pseudo-Riemannian symmetric spaces (see below for further discussion). 
They are also  holomorphic symmetric  domains in $\C^n$. Indeed, $\dS_n(\C)$ is the exterior of a ball in the projective space $\P^n(\C)$. It is strictly pseudo-concave. The ball of $\P^n(\C)$ is identified with the hyperbolic space $\H^n(\C)$, and then $\dS_n(\C)$ is the space of 
geodesic complex hypersurfaces of $\H^n(\C)$. 

 As for      $ \AdS_n(\C)$, it can be represented as  the open set $q <0$ of $\P^n(\C)$, where $q =- \vert z_0\vert^2 - \vert z_1\vert^2 + \vert z_2 \vert^2 + \ldots  + \vert z_n \vert^2$.

\subsubsection{Irreducible K\"ahler-Lorentz symmetric spaces}

Let  $M= G/H$ be  a homogeneous space. Call $p$ the base point $1.H$. The isotropy 
representation at $p$ is identified  with the adjoint representation  $\rho: H \to  \GL(\g/\h)$, where $\g$ and $\h$ are the respective Lie algebras of $G$ and $H$. The homogeneous space is of Hermite-Lorentz type if $\rho$ is conjugate to a representation in $\U(1, n)$
(where the real dimension of $G/H$ is $2n +2$).  

The space $G/H$ is symmetric if $-Id_{T_pM}$ 
belongs to the image of $\rho$. This applies in particular to the two following 
spaces: 
$$\C \dS_n = \SO^0(1, n+1) / \SO^0(1, n-1) \times \SO(2)$$
$$\C \AdS_n = \SO(3, n-1)/\SO(2) \times \SO^0(1, n-1)$$

\subsubsection*{Complexification} The isotropy representation of these two spaces is the 
complexification of the $\SO^0(1, n-1)$ standard representation in $\R^n$, i.e. its diagonal 
action on $\C^n = \R^n +i \R^n$; augmented with the complex multiplication by $\U(1) \cong \SO(2)$. 

If one agrees that  a complexification of a homogeneous space $X$ is  a homogeneous space $\C X$ whose  isotropy  is the complexification of that of $X$,  then $\C \dS_n$ and $\C \AdS_n$
appear naturally as complexification of $\dS_n(\R)$ and $\AdS_n(\R)$ respectively. In contrast, $\dS_n(\C)$ and $\AdS_n(\C)$ are the set of complex points of the same algebraic object as $\dS_n(\R)$ and $\AdS_n(\R)$. As another example, the complexification of $\mathbb S^n$ is not $\P^n(\C)$ but rather the K\"ahler Grassmanian space $\SO(n+2)/\SO(n)\times \SO(2)$?!

\subsubsection{List}  There are lists of pseudo-Riemannian {\bf irreducible} symmetric spaces, see for instance \cite{Besse, Berger, Kath}. (Here irreducibility  concerns isotropy, but for symmetric spaces, besides the flat case, the holonomy and isotropy groups coincide. In particular,  holonomy irreducible symmetric spaces are isotropy irreducible) . It turns out that the five previous spaces are all the K\"ahler-Lorentz (or equivalently  Hermite-Lorentz) ones.  \\
{\it Our theorem \ref{irreducible} below will give in particular  a non list-checking proof of this classification.}

 \subsubsection*{Acknowledgments}  We would like to thank Benedict Meinke for his  careful reading and valuable remarks on the article.

 \section{Results}

\subsubsection{Convention} It is sometimes a nuisance and with no real interest to deal with ``finite objects''.  We will say, by the occasion, that  some fact is true up to finite index, if it is not necessarily satisfied by the given group itself, say $H$, but for another one $H^\prime$ commensurable to it, that is $H \cap H^\prime$ 
has finite index in both. We also use ``up to finite cover''  for a similar meaning.

\subsubsection{Objective} Our aim  here is the study of isometry groups $\Iso(M, J, g)$ of Hermite-Lorentz manifolds. They are Lie groups acting holomorphically on $M$.   If $g$ were (positive definite) Hermitian, then $\Iso(M, J, g)$ acts properly on $M$, and is in particular compact if $M$ is compact. 

This is no longer true for $g$ indefinite. 

In the real case, that is without the almost complex structure, there have been many works tending to understand how and why the isometry group of a Lorentz manifold can act non-properly (see for instance \cite{ADZ, DMZ}). The Lorentz case is the simplest among all the pseudo-Riemannian cases, since, with its one negative sign, it lies as the nearest to the Riemannian case. For instance, the situation of signature $- - + \ldots +$ presents more formidable  difficulties. With this respect, the Hermite-Lorentz case seems as an intermediate situation, which besides mixes in a beautiful way pseudo-Riemannian and complex geometries.

\subsection{Homogeneous vs Symmetric}  We are going to prove facts characterizing these K\"ahler-Lorentz symmetric  spaces   by means of a homogeneity  hypothesis (as stated  in Theorems \ref{irreducible} and \ref{nonproper} below).

In pseudo-Riemannian geometry, 
it is admitted that, among homogeneous spaces, the most beautiful are those of constant sectional curvature, and then the symmetric ones, and so on... This also applies to pseudo-K\"ahler spaces, where the sectional curvature is replaced by the holomorphic sectional curvature.

In general,  being (just)  homogeneous is so weaker than being symmetric which in turn is weaker than having constant (sectional or holomorphic sectional) curvature.

For instance, Berger spheres are homogeneous Riemannian metrics on the 3-sphere that have non constant sectional curvature and are not symmetric. On the other hand different Grassmann spaces are irreducible symmetric Riemannian  (or Hermitian) spaces but do not have  constant sectional (or holomorphic) curvature.

Our first theorem says that in the framework of Hermite-Lorentz spaces, being homogeneous implies essentially symmetric!

\begin{theorem} \label{irreducible} Let $(M, J, g)$ be a Hermite-Lorentz almost complex space,  homogeneous under the action of a Lie group $G$. Suppose that    
the isotropy group $G_p$ of some point $p$ acts $\C$-irreducibly on $T_pM$, and $\dim_\C M >3$. Then $M$ is a global  K\"ahler-Lorentz  symmetric space, and it is isometric,
 up to a cyclic cover, to  $\Mink_n(\C)$, $\dS_n(\C)$, $\AdS_n(\C)$, $\C \dS_n$ or $\C \AdS_n$.
 
\end{theorem}

 $\bullet$ The content of the theorem is:
 
 1. Irreducible  isotropy $\Longrightarrow$  symmetric,  
 
 2. The list of Hermite-Lorentz symmetric spaces with irreducible isotropy are the five mentioned  ones. 
   This   fact 
 may be extracted  from Berger's classification of pseudo-Riemannian irreducible symmetric spaces. Here, we provide a direct proof.

 $\bullet$ In the (real) Lorentz case, there is a stronger version, which states that an isotropy irreducible homogeneous space has constant sectional curvature \cite{BoZ} (the fact that irreducible and symmetric implies constancy of the curvature was firstly observed in \cite{Cohen} by consulting Berger's list).

 $\bullet$ The theorem is not true in the Riemannian case. As an example of a compact  irreducible isotropy non-symmetric space, we have $M = G/K$, where $G = \SO(\frac{n(n-1)}{2})$, 
 and $K$ is the image of the representation of $\SO(n)$ in the space of trace free  symmetric 2-tensors on $\R^n$ (see \cite{Besse} Chap 7).

\subsection{Actions of semi-simple Lie groups} Let now $(M, J, g)$ be an almost Hermite-Lorentz manifold and $G$ a Lie group acting (not necessarily transitively)  on $M$ by preserving its structure.  We can not naturally 
make a hypothesis on the isotropy in this case, since it can be    trivial (at least for generic points).
It is however more  natural to  require 
  dynamical properties on the action. As discussed in many places (see for instance \cite{ADZ, DMZ}), non-properness of the $G$-action is a reasonable condition allowing interplay between dynamics and the geometry of the action. For instance, without it everything is possible; a Lie group $G$ acting  by left translation on itself can be equipped by any    type of tensors by  prescribing it  on the Lie algebra.  
  
  The literature contains many   investigations   on non-proper actions preserving Lorentz metrics \cite{A1, A2, ADZ, DMZ} and specially \cite{Kow}. We are going here to ask similar questions on the Hermite-Lorentz case. We restrict ourselves here to transitive actions, since the general idea, within this geometric framework, is that a $G$-non-proper action must have   non-proper $G$-orbits, i.e. orbits 
  with non-precompact stabilizer. The natural starting point is thus the study of non-proper transitive actions.

\begin{theorem} \label{nonproper}

Let $G$ be a non-compact simple (real) Lie group  of finite center  not   locally isomorphic to $\SL_2(\R)$,  $\SL_2(\mathbb C)$ or $\SL_3(\R)$. Let $G$  act non-properly transitively holomorphically and isometrically  on an almost complex  Hermite-Lorentz space $(M, J, g)$, with $\dim M >3$.    Then,  
 $M$ is a global  K\"ahler-Lorentz irreducible  symmetric space, and is isometric,
 up to a cyclic cover,  to    $\dS_n(\C)$, $\AdS_n(\C)$, $\C \dS_n$ or $\C \AdS_n$.

 \end{theorem}

 \subsection{Some Comments}

\subsubsection{Integrabilities} Observe that in both theorems, we do not assume a priori neither that   $J$ is 
integrable, nor $g$ is K\"ahler. 
  \subsubsection{The exceptional cases }
 The hypotheses $\dim M >3$ and $G$ different form $\SL_2(\R)$, $\SL_2(\C)$ and $\SL_3(\R)$ are due on the one hand to ``algebraic'' technical difficulties in proofs and on the other hand  to that    statements  become complicated in this cases. 
 
 As an example,   $\SL_2(\C)$ with its complex structure  admits a left invariant Hermite-Lorentz metric $g$ which is moreover  invariant by the right action of $\SL_2(\R)$. So, its isometry group is  $G= \SL_2(\C) \times \SL_2(\R)$ and its isotropy is $\SL_2(\R)$ acting 
 by conjugacy. On the  Lie algebra $g$ is defined as: $\langle a, b\rangle =  \tr (a\bar{b})$,
 where $a, b \in sl_2(\C) \subset \Mat_2(\C)$. This metric is not K\"ahler, neither symmetric, although the isotropy is $\C$-irreducible, and so Theorem \ref{irreducible}
does not apply in this case.

 In the case of $\SL_3(\R)$ one can construct an example of a left invariant  
 Hermite-Lorentz structure $(J, g)$, with $J$ non-integrable, invariant under the action by conjugacy of a one parameter group, and therefore Theorem \ref{nonproper} does not apply to the 
 $\SL_3(\R)$-case.   Notice on the other hand that, although $\SL_3(\R)$ is a not a complex Lie group, it admits left invariant complex structures. This can be seen for instance by observing that its natural action on $\P^2(\C) \times \P^2(\C)$ has an open orbit on which it acts freely. 
 We hope to come back to this discussion  elsewhere.

 \subsection{About  the proof}
  
 The tangent space at a base point of $M$ is identified to $\C^{n+1}$, and the isotropy $H$ to a subgroup of $\U(1, n)$. 
 
 ({\it Henceforth, in all the article, the complex dimension of the manifold $M$ will be $n+1$}).
 
 \subsubsection{Subgroups of $\U(1, n)$}  We state in Proposition \ref{subgroups}  a classification of such subgroups (when connected and non pre-compact ) into amenable and (essentially) simple ones.
 
 In fact  there  have been many works on the (recent) literature  about these groups, generally 
 related to the study of the holonomy of  pseudo-Riemannian and pseudo-K\"ahler spaces. Indeed, 
 A.  Di Scala   and  T.  Leistner  classified  connected irreducible Lie subgroups 
of $\SO(2, m)$ \cite{D-L} (which contains our case  $\U(1, n)$, for $m = 2n$). On the other hand,  the  case of non necessarily irreducible connected subgroups
of $\U(1, n)$ was   considered by A. Galaev in \cite{Ga} and  Galaev-Leistner in \cite{GL}. There, the authors used the term
``pseudo-K\"ahler of index 2'' for our ``K\"ahler-Lorentz'' here. 

A proof of Proposition \ref{subgroups}
Êcould be extracted from these references, but 
for reader easiness we give here our independent (and we think more geometric) proof! More important, in our proofs
(and hypotheses) of Theorems \ref{irreducible} and  \ref{nonproper}, we deal with non-necessarily connected groups (that is  the isotropy $H$ is not assumed to be connected)! The analysis of connectedness occupies in fact a large part of  \S \ref{sousgroupes}.

  \subsubsection{}   Regarding Theorem \ref{irreducible}, since it acts irreducibly,  the possibilities given 
 for  $H$ (more precisely its Zariski closure) are $\U(1, n)$, $\SU(1, n)$, $\U(1) \times \SO^0(1, n)$, and $\SO^0(1, n)$ (the last acts $\C$-irreducibly but not $\R$-irreducibly). 
 Geometric and algebraic manipulations yield
 the theorem, that   is the explicit possible $G$ and $H$, \S \ref{proof.theorem1}.  At one step of the proof, we show  that $M$ is  a symmetric space, but 
we do not refer to the  Berger's classification to get its form.

Here also, an  alternative, but highly more algebraic approach,  would be to use elements of the theory of reductive homogeneous spaces to  show that 
$M$ is symmetric, and in a next step to consult Berge's list, by  showing which members of it are K\"ahler-Lorentz spaces!  Here again, the most difficulty comes from the a priori  non-connectedness of $H$.

 \subsubsection{}   As for Theorem \ref{nonproper}, the idea is to apply Theorem \ref{irreducible} by showing that $H$ is irreducible (assuming it non-precompact and  $G$ simple). 
 
 - One starts proving that $H$   is big enough, \S \ref{preliminaries}. 
 
 -- If $H$ is  simple, irreducibility consists in  excluding  the intermediate cases $\SO^0(1, k) \subset H \subset \U(1, k)$, for $k<n$, \S  \ref{reductive}. Here again, 
 theory of reductive homogeneous spaces could apply, but not to the general non-amenable (non connected) case.

$\bullet$   However, the most delicate situation to exclude is the   amenable one,  \S \ref{non.reductive}. 
Observe in fact  that, in general,  homogeneous pseudo-Riemannian manifolds with a semi-simple Lie group may have, for instance, 
an abelian    isotropy.  Take for example $\dim H = 1$, that is $H$  a one parameter group, and assume its Lie sub-algebra $\frak h$ is 
non-degenerate (i.e. non-isotropic) in $\frak g$, the Lie algebra of $G$ endowed with its Killing form. The $G$-action on the quotient space $G/H$ will then preserve a pseudo-Riemannian metric given by the restriction of the Killing form on $\frak h^\perp$. (More complicated constriction are surely possible!)

  This part of proof, i.e that $H$ can not be amenable (in our Hermite-Lorentz case) might be  the essential mathematical (i.e. from the point of view of proof) 
 contribution of the present article. Very briefly, amenability allows one to associate  to any point of $M$
 a  lightlike complex hypersurface,
fixed by its isotropy group. Now,  the point is  to prove that this determines a foliation, that is two 
 such hypersurfaces are  disjoint or equal.  The contradiction will come from that the quotient space of such a foliation is
  a complex surface on which the group $G$ acts non-trivially, which was excluded by the hypothesis that $G$ is different from
  $\SL_2(\R)$,  $\SL_2(\mathbb C)$ or $\SL_3(\R)$.

  \section{Some preparatory facts}
 
  \label{preliminaries}

 $\C^{n+1}$ is endowed with the standard Hermite-Lorentz form
 $q_0 = - |z_0 |^2 + |z_1|^2 + \ldots + | z_n|^2$. The Hermitian product is denoted $\langle, \rangle$. 
  
  Recall that $u$ is lightlike (or isotropic) if $q(u) = 0$. A $\C$-hyperplane is lightlike if it equals the orthogonal $\C u^\perp$ of a lightlike vector $u$. 
  
  It is also sometimes useful to consider the  equivalent form $q_1=  z_0 \bar{z_n}  + \bar{z_0}z_n+ |z_1|^2 + \ldots + | z_{n-1}|^2$. 
  
  As usually, 
 $\U(q_0)$ is denoted $\U(1,n)$, and $\SU(1, n)$ its special subgroup. 
 
 The Lorentz group $\SO^0(1, n)$ is a subgroup of $\SU(1, n)$; it acts diagonally on $\C^{n+1} = \R^{n+1} + i\R^{n+1}$, by $A(x +iy) = A(x) + iA(y)$.

 \subsection{Some $\SO^0(1, n)$-invariant theory}

 \subsubsection*{Levi form}

  \begin{fact} \label{invariant.levi} 
  
  1. For $n >1$, there is no non-vanishing  $\SO^0(1, n)$-invariant anti-symmetric  form $\R^{n+1} \times \R^{n+1} \to \R$.

2. Let $n>1$ and  $b : (\R^{n+1} + i \R^{n+1}) \times (\R^{n+1} + i \R^{n+1}) \to \R$  be a $\SO^0(1, n)$-invariant anti-symmetric bilinear form. Then, up to a constant,  $b(u +iv, u^\prime + i v^\prime) = 
  \langle u, v^\prime \rangle - \langle v, u^\prime \rangle $. (That is, up to a constant,  $b$ coincides with the K\"ahler form $i(-dz_0 \wedge \bar{dz_0} + dz_1 \wedge \bar{dz_1} + \ldots + dz_n \wedge \bar{dz_n})$).

  \end{fact}

 \begin{proof}
 
 ${}$

 1. Let $b: \R^{n+1} \times \R^{n+1} \to \R$ be such a form, and $u$ a timelike vector, that is 
 $\langle u, u \rangle <0$. Thus, the metric on $\R u^\perp$ is positive and the action of 
 the stabilizer (in $\SO^0(1, n)$) of $u$ on it is  equivalent to the usual action of $\SO(n)$ on $\R^n$. The linear  form $v \in \R u^\perp \to  b(u, v) \in \R$ is $\SO(n)$-invariant, and hence vanishes (since its kernel is invariant, but the $\SO(n)$-action is irreducible). Thus $u$  belongs to  the kernel of $b$, and so is any timelike vector, and therefore $b= 0$

 2. Let now $b: (\R^{n+1} + i \R^{n+1}) \times (\R^{n+1} + i \R^{n+1}) \to \R$. From the previous point $b(u + i 0, v + + i0 ) = b(0 + iu, 0 + iv) = 0$. It remains to consider
 $b(u, iv)$.  It can be written $b(u, iv) = \langle u, A v \rangle$, for some $A  \in \End(\R^n)$, commuting with $\SO^0(1, n)$. By the (absolute) irreducibility of $\SO^0(1, n)$,  $A$ is scalar. (Indeed, by irreducibility, $A$ has exactly one eigenvalue $\lambda$ with 
 eigenspace the whole $\R^{n+1}$. If $\lambda$ is pure imaginary, then $\SO^0(1, n)$ 
 preserves a complex structure, but this is impossible (for instance hyperbolic elements of $\SO^0(1, n)$ have simple real eigenvalues). Thus $A$ is a real scalar).
 
  The rest of the proof follows.  
 \end{proof}

 \subsubsection*{K\"ahler form}

 \begin{fact} \label{no.3.form}

For $n>2$, there is no non-vanishing (real) exterior 3-form $\alpha $ on $\mathbb C^{n+1} (= \mathbb R^{n+1} + i \R^{n+1}$) invariant under the $\SO^0(1, n)$-action. 
\end{fact}

\begin{proof}

Let $\alpha$ be such a form. Let $e \in \R^{n+1}$ be  spacelike: $\langle e, e \rangle >0$, and  consider $\alpha_e= i_e\alpha$. First, $\alpha(e, ie, z) $ is a linear form on $\C e^\perp$
invariant under a group conjugate to $\SO^0(1, n-1)$, and hence vanishes. On $\C e^\perp$, $\alpha_e$ is a 2-form as in the fact above. It then follows that for any $u \in \C e$, and $v, w \in \C u^\perp$, $\alpha (u, v, w) = \phi(u)\omega(v, w)$, where $\omega $ is the K\"ahler form on 
$\C e^\perp$, $\phi: \C e \to \R$ is a function, necessarily linear. There is $u \in \C e$ such that $\phi(u) = 0$, and hence $u \in \ker \alpha$. This kernel is a $\SO^0(1, n)$-invariant space. If it is not trivial, then it has the form $\{au + biu, u \in \R^{n+1} \}$, where $a$ and $b$ are constant. But $\alpha$ induces a form on the quotient  $\C^{n+1}/ \ker \alpha$ which vanishes for   same reasons. Hence $\alpha = 0$.

 \end{proof}

\subsubsection*{Nijenhuis tensor}

\begin{fact} \label{integrability} \label{invariant.SO}   For $n > 2$,  there is no non-trivial   anti-symmetric bilinear form $\C^{1+n} \times \C^{1+n} \to \C^{1+n}$, equivariant under $\SO^0(1, n)$.

\end{fact}

   \begin{proof}

   Let $b:   \C^{n+1} \times \C^{n+1} \to \C^{n+1}$ be a $\SO^0(1, n)$-invariant  anti-symmetric form.
   
   Let us first consider the restriction of $b$ to $\R^{n+1}$. 
    Let $u, v  \in \R^{n+1}$ two   linearly independent  lightlike vectors and  $w  $ in the  orthogonal  space $\Span_\C(u, v)^\perp  \cap \R^{n+1}$. Consider 
   $H$   the subgroup 
  of $A \in \SO^0(1, n)$ 
 such that there exists $\lambda \in \R$,  $A(u) = \lambda u$, $A(v) = \lambda^{-1}v$, 
  and $A(w) = w$.   By equivariance $b(u, v)$ is fixed by $H$. But  $H$ is too big; its  fixed point set   is $\C w$. 
  Indeed its action on 
  $\Span_\C(u, v, w)^\perp$ is equivalent to the action of  $\SO(n-2)$ on $\R^{n-2}$. 
 Therefore,  $b(u, v) \in \C w$. But, since $n \geq 3$, we have freedom to choose
 $w$ in $\Span_\C (u, v)^\perp \cap \R^{n+1}$. Hence, $b(u, v) = 0$. 
 Last, observe that $\R^{n+1}$ is  generated by 
  lightlike vectors and  hence $b = 0$ on $\R^{n+1}$.
  
  One can prove in the same meaner that $b(u, iv) = 0$, for $u, v \in \R^{n+1}$. It then follows that $b = 0$.

    \end{proof}

\begin{remark} \label{remark.integrability} [Dimension 3] For $n = 2$, the vector  product $\R^{2+1} \times \R^{2+1} \to \R^{2+1}$
  is anti-symmetric and $\SO^0(1, 2)$-equivariant. One can equally define a vector product 
  on $\C^{2+1}$ equivariant under $\SU(1, 2)$. For given $u, v$, $u \wedge v$ is such that $\det (w, u, v) = \langle w, u \wedge v\rangle$ (here $\langle, \rangle$ is the Hermitian product on $\C^{2+1}$). 
   Observe nevertheless  that this vector product is  not equivariant under $\U(1, 2)$.

   \end{remark}

  \subsection{Parabolic subgroups} \label{parabolic} By definition, a maximal parabolic subgroup of $\SU(1, n)$ is the stabilizer of a lightlike direction. It is  convenient here to consider the form $q_1=  z_0 \bar{z_n} + \bar{z_0} z_n + |z_1|^2 + \ldots + | z_{n-1}|^2$. Thus, $e_0$
  is lightlike  and the stabilizer  $P^\prime$ of $\C e_0$  in $\U(1, n)$ consists of elements of the 
   form: 
   \[  \begin{pmatrix} \begin{matrix} a & 0& 0\\0&A&0\\0&0& \bar{a}^{-1}\end{matrix}
 \end{pmatrix}\]
where, $a \in \C^*$, $A \in \U(n-1)$, $\a \bar{a}^{-1} \det A = 1$.
 Thus $P$    is a semi-direct product 
  $S(\C^* \times \U(n-1)) \ltimes \sf{Heis}$. 
  
  The stabilizer $P$ of $\C e_0$ in $\SU(1, n)$ is $P^\prime \cap \SU(1, n) = S((\C^* \times \U(n-1)) \ltimes \sf{Heis})$

The   Heisenberg group is the unipotent radical of $P$ and consists of: 
 \[ \begin{pmatrix} \begin{matrix} 1& t&-\tfrac{\parallel t \parallel ^2}2 + is\\0&1&-\bar{t}\\0&0& 1\end{matrix}
 \end{pmatrix}\]
where $t \in \C^{n-1}$ and $ s \in \R$. 

We see in particular that $P$ is {\bf amenable}. Recall here that a topological group is amenable if  any  continuous action of it on
a compact metric space, preserves some probability measure.  For instance,  as it will  seen in \S \ref{sousgroupes},   our group $P$ will act on the boundary at infinity of the complex hyperbolic space and preserves there a Dirac measure. (In general, a Lie group is amenable iff its semi-simple part is compact).

  \subsection{Lightlike geodesic hypersurfaces} \label{light.geodesic} We will meet (especially in \S \ref{asymptotic}) special complex hypersurfaces $L \subset M$. We say that $L$ is {\bf lightlike} if for any $y \in L$, $T_yL $ is a lightlike complex hyperplane of $(T_yM, g_y)$. The kernel of $(T_yL, g_y) $ defines 
  a complex line sub-bundle $N$ of $TL$ (not necessarily holomorphic). The metric on 
  $TL/N$ is positive.
  
  We say that $L$ is (totally) {\bf geodesic} if for any $u\in TL$, the geodesic $\gamma_u$ tangent to $u$, is locally contained in $L$ (there exists $\epsilon$, such that $\gamma_u(]-\epsilon, + \epsilon[) \subset L)$. This is equivalent to invariance of $TL$ by the Levi-Civita connection;  if $X$ and $Y$ are vector fields defined in a neighbourhood of $L$, and are tangent to $L$ (i.e. $X(y), Y(y) \in TL$, for $y \in L$), then $\nabla_X Y(y) \in TL$, for $y \in L$.

Let us prove in this case that $N$ is parallel along $L$ and thus it is in particular integrable.  For this,  consider  three  vector fields $X, Y$ and $Z$ tangent to $L$, with  $X$  tangent to $N$.   
We have $\langle X, Z \rangle = 0$, and 
thus $0 = Y\langle X, Z \rangle = \langle \nabla_YX, Z \rangle + \langle X, \nabla_YZ \rangle= \langle \nabla_YX, Z \rangle$ (since $X$ is tangent to $N$). This is true for any $Z
$, and therefore  $\nabla_YX$ is tangent to $N$, which means that $N$ is a parallel 2-plane field. 

Denote by ${\mathcal N}$ the so defined foliation of $L$. The leaves are complex curves. Transversally, $\mathcal N$ is a {\bf Riemannian foliation}, that is,  there is a well defined projected 
Riemannian metric  on the leaf (local) quotient space $Q = L /{\mathcal N}$. Equivalently, 
the Lie derivative $L_Xh = 0$, where $h$ is the metric restricted to $L$ and $X$ is tangent to $N$. This is turn is equivalent to that, for any $Y$ invariant under the $X$-flow, i.e. $[X, Y] = 0$, the product $\langle Y, Y \rangle$ is $X$-invariant. 
To check  this, observe that $X  \langle Y, Y \rangle = \langle \nabla_XY, Y \rangle =
\langle \nabla_YX, Y \rangle= 0 $, since as we have just proved, $N$ is parallel (that is 
$\nabla_Y X$ is tangent to $N$).

\begin{corollary} \label{preservation} Let $f$ be an isometry of $M$ preserving $L$ and fixing a point $x\in L$. Assume 
$D_xf \in \GL(T_xM)$ is unipotent (i.e. $D_xf-Id$ is nilpotent). Then $f$ preserves (individually) each  leaf of $\mathcal N$.  

\end{corollary}

\begin{proof} $f$ acts as an isometry $\hat{f}$ of  the (local) quotient space $L/ {\mathcal N}$  endowed  with its projected   Riemannian metric.  The   derivative $D{\hat{x}}\hat{f}$
at the projection of $x$ is unipotent. But the orthogonal group
$O(n)$ contains no non-trivial unipotent elements. Therefore, 
$D_{\hat{x}}\hat{f} = Id_{T_{\hat{x}} Q}$, and hence as a Riemannian isometry, $\hat{f} = {Id}_Q$ (of course, we are tacitly assuming   everything   connected).

\end{proof}

 \section{Subgroups of $\U(1, n)$}
 \label{sousgroupes}
 
 The following proposition says roughly that,   up to compact objects, a subgroup  of $\SU(1, n)$ 
 is either contained in a parabolic group, or conjugate to one of the standard subgroups    $\SO^0(1, k)$ or 
 $\SU(1, k)$.

 \begin{proposition} \label{subgroups} Let $H$ be a {\bf non-precompact connected} Lie  subgroup of    $\SU(1, n)$ (i.e. its closure in not compact). Then:
 \begin{enumerate}
 
 \item
  $H$ is amenable iff it preserves a lightlike hyperplane (that is, by definition,  $H$ is contained in a maximal parabolic subgroup).

\item  In opposite, if $H$  acts $\C$-irreducibly on $\mathbb C^{n+1}$ (there is no 
non-trivial complex invariant subspace), then $H$   equals $\SO^0(1, n)$ or $\SU(1, n)$ (See 
also \cite{D-L}).

\item  In the  general (intermediate) case, when  $H$ is not amenable,  it  acts $\R$-irreducibly on some    non-trivial  subspace $E$, such that:

 \begin{enumerate} 
\item    Either  $E$ is totally real, and up to a conjugacy in $SU(1, n)$, $E = \R^{k+1} \subset \R^{n+1}Ê\subset \R^{n+1} + i \R^{n+1} = \C^{n+1}$, and up to finite index, $H$ is a product   $C \times SO(1, k)$, for $C$ a pre-compact subgroup acting trivially on $E$.

\item  or $E$ is a complex subspace, and up to   conjugacy in $\SU(1, n)$, $E = \C^{k+1}$, and $H$  is  $C \times \SU(1, k)$, where $C$ is as previously.

 . 
  
 \end{enumerate}
  
 \end{enumerate}

 \end{proposition}

 \begin{remark} As it will be seen from its proof, this classification naturally generalizes to connected subgroups of all simple Lie groups of rank 1.  The proof uses essentially one  standard result from  simple Lie groups theory,
  du to Karpelevich \cite{Kar}  and Mostow \cite{Mos}. It states that a given Cartan decomposition of a Lie subgroup extends to a Cartan decomposition of the ambient simple Lie group.  An essentially geometric (algebraic free) approach is also available in the case of $\SO^0(1, n)$, see \cite{BoZ, DO}.
 
 --  Observe finally that we do not assume $H$ to be closed.
 
 \end{remark}
 
 \begin{proof} Let $H \subset \SU(1, n)$ be as in the proposition.

\subsubsection{Hyperbolicity} Let $G = \SU(1, n)$, $K = \U(1, n-1)$ and consider $X = G/K = \H^n(\C)$ the associated Riemannian symmetric space.  We let $\SU(1, n)$ act on 
  the (visual Hadamard)  boundary $\partial_\infty X$,  
  which  is identified to the space of  complex   lightlike directions of  $\C^{1+n}$. (See 
  \cite{Goldman} to learn about the geometry of $\H^n(\C)$).

By definition,   a maximal parabolic subgroup $P$ is the stabilizer of a lightlike direction, or equivalently a point of $\partial_\infty X$. 
From \S \ref{parabolic}, $P$ is amenable (the fact that maximal  parabolic groups are amenable characterizes rank 1 groups). Therefore, any group fixing a point at $\partial_\infty X$ is amenable.

We have to prove conversely that a non-precompact connected  amenable group fixes some point at $\partial_\infty X$. 

Recall that elements of $\SU(1, n)$ are classified 
into elliptic, parabolic or hyperbolic. 

An isometry   is elliptic if it  fixes some point in $X$,  and thus  lies in  its stabilizer which is a compact subgroup of $\SU(1, n)$. 
Conversely, any compact subgroup of $\SU(1, n)$ fixes some point of $X$. Thus an element is elliptic iff it generates a precompact subgroup.

A parabolic element has exactly one fixed point in $\partial_\infty X$, and a hyperbolic one has two fixed points. Furthermore, in both cases, by iteration, all points of $\partial_\infty X$ tend to these fixed points. In particular, in both cases, any invariant measure has support a set $F$ with 
cardinality  $\# (F) \leq 2$.

Now, $H$ is amenable and hence leaves invariant  a probability measure $\nu$  on $\partial_\infty X$. 

If $H$ contains a parabolic or a hyperbolic element, then the support of $\nu$  consists  of a set $F$ of one or two points, and hence, $H$ preserves  such $F$.
  If $F$ has cardinality 2, since $H$ is assumed to be connected, it fixes  each of the points of $F$. Therefore, in all cases, $H$ has a fixed point in $\partial_\infty  X$.

In order to prove that indeed $H$ contains a parabolic or hyperbolic element, one uses the classical fact that a non-compact connected Lie group
contains some non-precompact one parameter group (see for instance \cite{D'A} for a proof of this). No element of such a one parameter group can be elliptic.

This completes the proof of (1) in the proposition.  (Actually, a more self-contained proof, say without using this fact on non-compact Lie groups,  is available, but needs more details!).

\subsubsection{Non-amenable case}
  $H$  is  a semi-direct product $(S \times C) \ltimes R$ (up to  finite index) where $S$ is semi-simple with no compact factor, $C$ is compact semi-simple, and $R$ is the (solvable) radical.  Observe that   $R$ must be  precompact. Indeed, if not, from the above proof, the fixed point set $F$ of $R$ in 
  $\partial_\infty X$ has cardinality 1 or 2.  Since $R$ is  a normal subgroup of 
  $H$, $F$ is preserved by $H$. By connectedness,  $H$  fixes  each of the points of $F$, hence $H$ is contained in a parabolic group, and is thus amenable.

  This implies that the semi-direct product is in fact direct, up to finite index ($H$ acts by conjugacy on the the compact torus $\bar{R}$. But the identity component of the automorphism group of a torus is trivial). Let us say that $H$ is a product $S \times C^\prime$ where $C^\prime$ is precompact.  We now investigate $S$ and come back later on  to $S \times C^\prime$. Observe first that $S$ is simple. Indeed, if $S = S_1 \times S_2$, then any non-elliptic $f \in S_2$, will centralize $S_1$, which implies $S_1$ has a fixed point at $\partial_\infty X$ and hence amenable.

\subsubsection{Simple Lie subgroups}  In order to understand the geometry of $S$, we investigate the symmetric space $X$ (rather than its boundary as in the previous  case).
Let $p$ be a base point, say that fixed by   the maximal compact $K$. We have a Cartan decomposition of the Lie algebra of $G$: $\g = \p \oplus  \k$, where $\k$ is the Lie subalgebra of $K$, and $\p$ is (the unique) 
 $K$-invariant  supplementary  space of $\k$ in $\g$.  Geometrically, $\p$  is identified with 
 $T_pM$, and if $u \in \p$, the orbit $\exp (tu).p$ is the geodesic of $X$ determined by $u$. 
 
 More other properties are: $K$ acts irreducibly on $\p$, and $[\p, \p]Ê= \k$. 
 
 If $S \subset G$ is a simple Lie subgroup, then Karpelevich-Mostow's theorem \cite{Kar, Mos} states, up to a conjugacy 
 in $G$ (or equivalently a modification of the base point), we get a Cartan decomposition by taking intersection: $\s = \s \cap \p \oplus \s \cap \k$. 
 
 Observe that $\s \cap \p$ determines
 $\s$, since $\s \cap \k = [\s \cap \p, \s \cap \p]$. 
 
 In our case, $\p = T_p X$ is identified to $\C^n$. The subspace $ E= \s \cap \p$ is either complex or totally real, since $ E \cap i E$ is $S\cap K$-invariant, and this last group acts irreducibly on 
 $E$. Now, $K = \U(n)$ acts transitively on the set of totally real (resp. complex) planes of a given dimension $k$. Thus, up to conjugacy, $\s\cap \p$ is the canonical $\R^{k+1}$
 or $\C^{k+1}$ in $\C^n$.  Candidate for $\s$ in these cases are the Lie algebras of the standard subgroups $\SO^0(1, k)$ or $\SU(1, k)$, respectively.  But 
since $\s\cap \p $ determines completely $\s$, there are the unique possibilities.

% it follows that 
 %$S$ is conjugate to one of the standard subgroups $\SO^0(1, k)$ or $\SU(1, k)$. 

 \subsubsection*{End} We have thus proved (2) and (3) of the proposition  at the group level: $H$ is conjugate 
 in $\SU(1, n)$ to $S \times C$, with $S= \SO^0(1, k)$ or $\SU(1, k)$. Since the  precompact factor 
 $C$ commutes with the non-compact $S$,  it is contained in $\SO(n-k)$ or $\SU(n-k)$, respectively. In particular $H$ preserves $\R^{k+1}$ or $\C^{k+1}$. This completes the proof of the proposition.   
 \end{proof}

 \begin{corollary} Let $L$ be a  subgroup of   $\SU(1, n)$ (not necessarily connected or closed) acting $\C$-irreducibly on $\C^{n+1}$. Then,   the identity component of its Zariski closure  equals $\SU(1, n)$ or  $\SO^0(1, n)$. 
 If the identity component $L^0$ is not trivial,  
 $L$ itself equals $\SU(1, n)$ or $\SO^0(1, n)$.

 \end{corollary}

 \begin{proof}  Let $L^{Zar}$ be the Zariski closure of $L$. It is a closed subgroup of $\SU(1, n)$ with  finitely many connected components. It is non-compact, since otherwise $L$ will be precompact and  can not  act irreducibly. 
 
 The identity component $H$ of $L^{Zar}$ is non-precompact too.
  If $H$ is amenable, then its fixed point set $F$ in $\partial_\infty X$ is preserved by $L$, since $L$ normalizes $H$. If $F$ consists of one point,  then $L$ fixes it, seen as a lightlike direction in $\C^{1+n}$ contradicting the fact that it acts irreducibly. 
  
  If $F$ consists of two lightlike directions, then $L$ preserves the (timelike) 2-plane that they generate in $\C^{1+n}$, again contradicting its irreducibility. 
  
  We infer from this that $H$ is not amenable. Apply Proposition \ref{subgroups} to get that $H$ is essentially $\SO^0(1, k)$ or $\SU(1, k)$. Our group $L$ itself is then contained in the normalizer of one of these groups. On easily sees that such a normalizer can not act 
  irreducibly unless $k = n$ (for instance the normalizer of $\SO^0(1, k)$ or 
  $\SU(1, k)$  preserves the space  of their  fixed vectors which is non-trivial for $k <n$). .
  \end{proof}
  
 We have more:
 
 \begin{corollary} \label{Zariski} Let $L$ be a  subgroup of   $\SU(1, n)$. 
If $L$ is non-amenable, then its Zariski closure contains a copy of $\SO^0(1,  k)$ or $\SU(1, k)$ for some $k >0$. If furthermore $L^0$ is non- pre-compact, then $L$ itself  contains $\SO^0(1, k)$
or $\SU(1, k)$. (Of course $\SO^0(1, k) \subset \SU(1, k)$, but we prefer our formulation here for a later use ). 
 
 \end{corollary}
 
 \begin{proof}  The first part is obvious.
 
 For the second one,   it suffices  to show that 
 $L^0$ is non-amenable.  But this $L^0$ is normalized by 
 $L^{Zar}$ ($L$ itself normalizes  $L^0$ and by algebraicity, $L^{Zar}$ too preserves it).
% The only one point to justify is that if $L^0$ is non-precompact, then it is non-amenable 
%(assuming  $L$  non-amenable). If not the Zariski closure $L^{Zar}$ will normalize ${L^0}$. 
But $\SO^0(1, k)$ or $\SU(1, k)$  can normalize no amenable non-compact connected subgroup of $\SU(1, n)$. 
 \end{proof}

\subsection{Subgroups of $\U(1, n)$} \label{subgroups.U} 

We will now deal with    subgroups $L$ of  of $\U(1, n) = \U(1) \times \SU(1, n)$.  The following lemma 
will help to understand them and maybe has its own interest:

\begin{lemma} Let $L$ be a subgroup of $\mathsf{U}(1, n)$.

% that is, non-precompact and does not act $\C$-irreducibly. Then:

1) If  $L$ is non-precompact, then:

(i)  either $L$ preserves a unique lightlike direction

(ii) or $L$ preserves a unique timelike  (i.e. on which the restriction of $q_0$ on it is of Hermite-Lorentz type) 2-plane, and also
each of the two lightlike directions inside it.

(iii) or $L$ preserves a unique timelike subspace on which it acts $\C$-irreducibly.

In all cases, this lightlike direction, or timelike subspace are invariant under the normalizer of $L$ in $\U(1, n)$. 

2) $L$ acts irreducibly iff $L \cap \mathsf{SU}(1, n)$ acts irreducibly.

\end{lemma}

\begin{proof} 
${}$

1) If $L$ acts $\mathbb C$-irreducibly, then we are done. So, assume it preserves some  proper subspace $E$, and thus also $E^\perp$. 
If $E$ is degenerate, then $E \cap E^\perp$ is an invariant lightlike direction. If $E$ is spacelike then $E^\perp$ is timelike, and vice versa.

Assume, we are not in  case (i), so either there is no invariant lightlike direction at all, or there are many. In all cases, we can find 
an invariant timelike subspace (by taking   sums if there are many lightlike directions).  

Let $E$ be an invariant timelike subsapce of minimal dimension.
Let us prove that either $E$ is irreducible, or we are in case (ii). 

Assume there exists  $E^\prime$ a proper   invariant subspace of $E$.  By definition, neither
 $E^\perp \cap E$ nor ${E^\prime}^\perp \cap E$ are timelike, and thus $E^\prime$ is degenerate, and hence $D = E^\prime \cap {E^\prime}^\perp$ is an invariant lightlike  direction. It is not unique, because we are not in case (i). So, there is another similar one $D^\prime$. It must be contained in $E$ since otherwise its projection on $E$ would be an invariant timelike direction. Let $P = D \oplus D^\prime$. By minimality,  $E = P$.  To show that we are in case (ii), let us 
prove uniqueness of $E$. If $D^{\prime \prime}$ is another invariant direction not in $P$, then its projection on $P$ will give a  timelike invariant direction. This implies
that the action on $P$ is equicontinuous, but since $P^\perp$ is spacelike, the group $L$ will be precompact in this case. 

It remains to consider the case where $E$ is irreducible, and show it is unique. Assume by contradiction that $E^\prime$ is analogous to $E$. Let
$R = E^\perp \oplus {E^\prime}^\perp$. The $L$-action on $R$ is equicontinuous (since both $E^\perp$ and ${E^\prime}^\perp$
are spacelike).  Let $P = R \cap E$. Then $P \neq 0$, unless $E = E^\prime$. Furthermore  $P \neq E$ since otherwise $E = \mathbb C^{1+n}$ and 
$L$ will be precompact. This contradicts the irreducibility of $E$. 

\medskip

 2)  Assume $L$ irreducible. 
 
 Consider the 
projections  $\pi_1$ and $\pi_2$ of 
 $\mathsf{U}(1, n) $ onto $ \mathsf{U}(1)$ and  $\mathsf{SU}(1, n)$ respectively. 
 
If $L \cap \mathsf{SU}(1, n) = \{1\}$, then $\pi_1 $ sends injectively  $L$ in $\mathsf{U}(1)$ and hence
$L$ is abelian, and  can not act irreducibly.  

Observe   that  $\pi_2(L)$ acts $\C$-irreducibly. Indeed, $L $ is contained in 
$\U(1) \times \pi_2(L)$ and $\U(1)$ preserves any $\mathbb C$-subspace.

Observe also that $\pi_2(L) $ and $L \cap \mathsf{SU}(1, n)$ normalizes each one the other, and that 
the commutator group $[\pi_2(L), \pi_2(L)]$ is contained in $L \cap \mathsf{SU} (1, n)$.  
If 
$L \cap \mathsf{SU}(1, n)$ is not precompact, then the previous step implies it is irreducible.  Finally, from
Corollary \ref{Zariski},  one infers that the commutator group of an irreducible subgroup of $\mathsf{SU}(1, n)$
 is not precompact, and therefore $L \cap \mathsf{SU}(1, n)$ acts irreducibly.

\end{proof}

\begin{corollary} \label{Zariski.U} Let $L $ be a subgroup of $\U(1,n)$ acting irreducibly 
on $\C^{n+1}$. Then, its Zariski closure contains $\SO^0(1, n)$ or $\SU(1, n)$.

If furthermore $L^0 \neq 1$,  and its Zariski closure does not contain $\SU(1, n)$, then, 
either, $L^0$ equals $\SO^0(1, n)$, or $L^0 \supset \U(1)$.

\end{corollary}

\begin{proof}

The first part is obvious from the discussion above, let us prove the second one. In this case
$L $ is a subgroup of $\U(1) \times \SO^0(1, n)$. 

--    If $L^0 \cap \SO^0(1, n) \neq 1$, then its equals $\SO^0(1, n)$ by irreducibility of 
$L \cap \SO^0(1, n)$, in particular $L \supset \SO^0(1, n)$. Furthermore, if a product  $ab$,  $a \in \U(1)$, 
$b \in \SO^0(1, n)$ belongs to $L$, then $b\in L$, that is $\pi(L) = L \cap \U(1)$. This 
last group is either $\U(1)$ (in which case $L = \U(1) \times \SO^0(1, n)$), or 
totally discontinuous, in which case $L^0 = \SO^0(1, n)$.

-- Assume now that $L^0 \cap \SO^0(1, n) = 1$. Let $l^t= a^tb^t$ be a one parameter group 
in $L^0$, and  $c \in L \cap \SO^0(1, n)$. The commutator $[c, a^tb^t]$ equals
$[c, b^t]$. This is a one parameter group in $L \cap \SO^0(1, n)$, and hence must be 
trivial. 
But since we can choose $c$ in a Zariski dense set in $\SO^0(1, n)$, the one parameter group 
$b^t$ must be trivial. This means that $l^t \in \U(1)$, and hence $L^0 \supset \U(1)$.

\end{proof}

 \section{Proof of Theorem \ref{irreducible}}
 \label{proof.theorem1}
 
 Let $(M, J, g)$ be an almost complex Hermite-Lorentz space on which a group $G$ acts transitively with $\C$-irreducible isotropy. 
 
 Let $p$ be a base point of $M$, and call $H$ its isotropy group in $G$. The tangent space 
 $T_pM$ is identified to $\C^{1+n}$ and $H$ to a subgroup of $\U(1, n)$. By hypothesis $H$ acts $\C$-irreducibly on $\C^{1+n}$.

 The first part of Theorem \ref{irreducible}, that is $J$ is integrable and $g$ is K\"ahler will be proved quickly.  Indeed, by  Corollary \ref{Zariski.U} the Zariski closure of $H$ (in $\U(1, n))$
 contains $\SO^0(1, n)$.

\subsubsection*{K\"ahler Character}  Let $\omega$ be the K\"ahler form of $g$. Its differential at $p$,  $\alpha = d \omega_p$ is an $H$-invariant 3-form on $\C^{n+1}$. By  Corollary  \ref{Zariski.U}, $\alpha$ is 
 $\SO^0(1, n)$-invariant.
 By    Fact \ref{no.3.form}, $\alpha= 0$, that  is,  $M$ is K\"ahler. 
 
\subsubsection*{Integrability of the complex structure} The (Nijenhuis, obstruction to) integrability tensor at $p$ is  an $\R$-anti- symmetric bilinear vectorial form $\C^{1+n} \times \C^{1+n} \to \C^{1+n}$. The same argument, using Fact \ref{integrability} yields its vanishing, that is $J$ is integrable. 

% of this tensor.

%, for $n \neq 2$. The integrability of the almost complex structure in the case 
 %$\dim M = 3$, will the treated  just below (\S  \ref{dimension3}).
 
 %at the end of this \S. 

\begin{remark} Observe that we need $\dim M >3$ in order to apply Facts \ref{no.3.form} and \ref{integrability}.

\end{remark}

\subsubsection{Classification} The rest of this section is devoted to the identification of $M$ 
as one of theses spaces: $\Mink_{n+1}(\C)$, $\dS_{n+1}(\C)$, $\AdS_{n+1}(\C)$, $\C \dS_{n+1}$
or $\C \AdS_{n+1}$ (up to  a central cyclic cover in some cases).

%as a K\"ahler Lorentz space

\subsubsection{The identity component $H^0$}\label{Indentity.non.trivial} Let us prove that $H^0 \neq 1$. If not $G$ is a covering of $M$, in particular $T_pM \cong \C^{n+1}$ is 
identified to the Lie algebra $\g$, and $H$ acts by conjugacy. The bracket is an $\R$-bilinear form like the integrability tensor, and hence vanishes, that is $\g$ is abelian. This 
contradicts the fact that $H$ acts non-trivially by conjugacy. Therefore 
$H^0$ is non-trivial. 
Applying   subsection \ref{subgroups.U}, we get  three possibilities:

1. The Zariski closure of $H$ contains $\SU(1, n)$

2. $H^0 = \SO^0(1, n)$

3. $H^0$ contains $\U(1)$. 

\subsection{Case 1:   the Zariski closure of $H$ contains $\SU(1, n)$} 
The holomorphic sectional curvature at $p$ is an $H$-invariant  function on the open subset in $\P^n(\C)$ of non-lightlike  $\C$-lines of  $\C^{n+1}$. But $\SU(1, n)$ acts transitively 
on this set. It follows that this holomorphic sectional curvature is constant. Therefore 
$M$ is a K\"ahler-Lorentz manifold of constant holomorphic sectional curvature, and thus 
$M$ is locally isometric to one the universal spaces $\Mink_{n+1}(\C)$, $\dS_{n+1}(\C)$
or $\AdS_{n+1}(\C)$  \cite{Kob-Nom, Romero1}.  We will see below (\S \ref{global.symmetry})  that $M$ is (globally) isometric to 
$\Mink_{n+1}(\C)$, 
$\dS_{n+1}(\C)$ or to a cover of $\AdS_{n+1}(\C)$.

%is a central cyclic 

%It is algebraic, and thus invariant by $\SU(1, n)$, which acts   transitively, and hence the %holomorphic sectional curvature is constant. 

 %In order to  complete the proof of Theorem \ref{irreducible}, it suffices to prove the %following:
 
 %\begin{lemma} Let $M$ be a $G$-homogeneous K\"ahler-Lorentz of constant holomorphic %curvature, with an irreducible isotropy $H$. 

% Then, $M$ equals a universal K\"ahler-Lorentz space $X$ of constant holomorphic curvature. 

%\end{lemma}

\subsection{Case 2:     $H^0 = \SO^0(1, n)$} \label{Case.SO} 
  The final goal here is to show that $M$ is $\Mink_{n+1}(\C)$. 
First we replace $M = G/H$
by $G/H^0$ which enjoys all the properties of the initial $M$. In other words, we can assume $H = H^0 = \SO^0(1, n)$.

\subsubsection*{Invariant distributions}  $\SO^0(1, n)$ acts $\C$-irreducibly but not
$\R$-irreducibly. We set a $G$-invariant distribution $S$ on $M$ as follows.
Define   $S$ to be equal to $\R^{n+1}$ at $p$. For $x =gp$, define $S_x = D_pg (S_p)$. This does not depend on the choice of $g$ since $S_p$ is $H$-invariant.

The orthogonal distribution  $S^\perp$  is   in fact determined similarly 
by means of the $H$-invariant space $i\R^{n+1}$.

\subsubsection*{Integrability of distributions}

 The obstruction to the  integrability of $S$ is encoded in 	 the anti--symmetric {\bf  Levi} form $II: S \times S \to S^\perp$, where  $II(X, Y) $ equals the projection on $S^\perp$ of $[X, Y]$, for $X$ and $Y$   sections of $S$. 
   At $p$, we get an anti-symmetric bilinear form $  \R^{n+1} \times \R^{n+1} \to \R^{n+1}$, equivariant under the  $\SO^0(1, n)$-action.
   By  Fact \ref{invariant.SO}, this must vanish and hence $S$ and analogously $S^\perp$ are
   integrable.

   We denote by $\mathcal S$ and  $\mathcal S^\perp$ the so defined foliations.

  Observe that since $G$ preserves these foliations, then each leaf of them is homogeneous. If $F$ is such a leaf,  $x, y \in F$, and  $g \in G$ is such that $y = gx$, then $g$ sends the distribution at $x$ to that at $y$, and hence, $gF = F$.
 
Leaves of $\mathcal S $ or $\mathcal S^\perp$ are (real)  homogeneous  Lorentz manifolds   with (maximal) isotropy $\SO^0(1, n)$. They are easy to handle du to the following fact, the proof of which is standard: 

\begin{fact} \label{isotropy.maximal} Let $F = A/B$ be a homogeneous Lorentz manifold of dimension $n+1$ such 
that the action of $B$ on the quotient $\mathfrak a/ \mathfrak b$ of Lie algebras 
is equivalent to the standard action of $\SO^0(1, n)$ on $\R^{n+1}$. Then $F$ has constant sectional curvature. If $F$ is flat, then $F = \Mink_{n+1}$ and $A = \SO^0(1, n)
\ltimes \R^{n+1}$. If $F$ has positive curvature then it equals $\dS_{n+1}$ and 
$A = \SO^0(1, n+1)$. Finally, in the negative curvature case, $F$ is a cover of 
$\AdS_{n+1}$, and $A$ covers $\SO(2, n)$.

\end{fact}

% (their dimension is $1+n$). It follows that they have constant sectional curvature. The %group 
% $L$ preserving the  leaf $F = \mathcal S_p$  contains the isotropy. This group is thus the %full isometry group and $F$ is a universal lorentz space of constant sectional curvature. 
 
 Let $A$ be the stabilizer of $\mathcal S_p$.
If leaves of $\mathcal S$ are not flat, then $ A $ is $\SO^0(1, n+1) $ in case of positive curvature,  and $ A = \SO(2, n)$
 in the negative curvature case. Consider   the (local) quotient space $Q= M / {\mathcal S}$, it has dimension $n+1$. The group $ A$ acts by fixing $F$, seen as a point of $Q$. But 
 $\SO^0(1, n+1)$  and  $\SO(2, n)$) have  no linear representation of dimension $n+1$. Therefore,  $A$  acts trivially on the tangent space $T_F Q$.  But this tangent space is identified to $S_p^\perp$. There, $\SO^0(1, k)$,  as an isotropy subgroup,  
 acts non-trivially. This contradiction  implies that the leaves of $\mathcal S$ and analogously $\mathcal S^\perp$ are flat.
 
 %the fact that $L \supset \SO(1, k)$.

%$\bullet$  

Now, we  need   to study further the geometry of our foliations. We claim that their leaves are in fact totally geodesic. Indeed, there  is a symmetric Levi form measuring the obstruction of  geodesibility.    More exactly, it is given by $II^*(X, Y)= $ the orthogonal projection of the covariant derivative
 $\nabla_X Y$. From \ref{invariant.SO}, since equivariant symmetric bilinear forms do not exist, 
 the foliations $\mathcal S$ and $\mathcal S^\perp$ are geodesic. It is classical   that the existence 
 of a couple of orthogonal geodesic foliations implies a metric splitting of the space, see for example \cite{Kob-Nom} about the proof of  the  de Rham decomposition Theorem (one starts  observing that . That is, at least locally, $M$ is isometric to the product $\mathcal S_{p} \times \mathcal S^\perp_{p}$. 
 In particular,  $M$ is a flat Hermite-Lorentz manifold, that is $M$ is locally isometric to $\Mink_{1+n}(\C)$. 
 
 One can moreover prove that $G$ is a semi-direct product $\SO^0(1, n) \ltimes \R^{n+1}$
 and $M = \Mink_{n+1}(\C)$ (see   \S \ref{global.symmetry} below for details in 
  a similar situation).
 %(it acts transitively on $\Mink_{1+n}(\C)$ with isotropy $\SO(1, n)$). 

 %However, in the isometry group
% $\Iso(\Mink_{1+n}) = \U(1, n) \ltimes \C^{1+n}$, there is no semi-simple group, with a %somewhere open orbit. Therefore, $M$ can not be homogeneous under the action of a semi-%simple
% group. 
% $\Box$

\subsection{Case 3:  $ \U(1) \subset H \subset   \U(1) \times \SO^0(1, n) $} The goal here is
to prove that $M$ is   flat or isomorphic to one of the two  spaces $\C \dS_n = \SO^0(1, n+1) / \SO^0(1, n-1) \times \SO(2)$
 or $\C \AdS_n = \SO(3, n-1)/\SO(2) \times \SO^0(1, n-1)$.

The crucial observation is that   $M$ is a (pseudo-Riemannian)  symmetric space, that is there exists $f \in G$, such that 
  $D_pf = -Id_{T_pM}$. Indeed, $-Id \in \U(1)$. 
  
   There is a de Rham decomposition of $M$ into 
 a product of  a flat factor and irreducible symmetric spaces. In our case, there exists  a subgroup of the isotropy  that  acts irreducibly. It follows that $M$ is either flat, or 
 irreducible. There is nothing to prove in the first case, we will therefore assume that 
 $M$ is irreducible. 
   We can also assume that $G$ is the full 
  isometry group of 
  $M$ (the hypotheses in Theorem \ref{irreducible} on the $G$-action are also valid for the full isometry group). It is known that isotropy groups  of symmetric spaces have finitely many  connected components. Thus, up to a finite cover (say assuming it connected),  $H$ must be
  $\U(1) \times \SO^0(1, n)$. 
  
  Consider a Cartan decomposition $\g = \h + \p$, where $\p$ is identified with $\C^{n+1}$.
  Consider the bracket $[, ]:  \p \times \p \to \h = so(1, n) + u(1)$. 
  
  Its second component 
  is a $\SO^0(1, n)$-invariant anti-symmetric scalar bilinear form $\alpha: \C^{n+1} \times \C^{n+1} \to u(1) = \R$. 
  By Fact \ref{invariant.levi}, $\alpha $ vanishes on $\R^{n+1}$, that is if 
  $X, Y \in \R^{n+1}$, then $[X, Y] \in so(1, n)$. On the other hand, $SO(1, n)$ preserves
  $\R^{n+1}$, and hence if $T \in so(1, n)$ and $Z \in \R^{n+1}$, then $[T, Z] \in \R^{n+1}$.
  
  Summarizing, if $X, Y, Z \in \R^{n+1}$, then $[[X, Y], Z] \in \R^{n+1}$. It is known that 
this implies that $\R^{n+1}$ determines a totally geodesic submanifold, say $F$. It has 
dimension $n+1$ and isotropy $\SO^0(1, n)$. From Fact \ref{isotropy.maximal}, $F$ is a Lorentz space of constant curvature. It can not be flat since in that  case, the bracket $[, ]$
vanishes on $\R^{n+1} $, but this implies it vanishes on the whole of $\C^{n+1}$. So $M$ is   
the de Sitter or the anti de Sitter space.  

The two cases are treated identically, let us  
assume $F = \dS_{n+1}$. Its isometry group 
$\SO^0(1, n+1)$ is thus contained in $G$. 

The goal now is to show that $G = \SO^0(1, n+2)$. For this, we consider the homogeneous space $N = G/\SO^0(1, n+1)$. Since we know the dimensions of $G/ \U(1) \times \SO^0(1, n)$ and $\SO^0(1, n+1)/\SO^0(1, n)$, we can compute that of $G/\SO^0(1,n+1)$, and find it equals $n+2$. 

Thus $\SO^0(1, n+1)$ has an isotropy representation $\rho$  in the $(n+2)$-dimensional space $E$, 
 the tangent space at the base point of $N$.
 %G/\SO(1, n+1)$.  
 In a  direct  way, we prove that this is the  
 standard representation of $\SO^0(1, n+1)$ in $\R^{n+2}$. For this, we essentially use that 
 $\rho$ restricted to $\SO^0(1, n)$ is already known. 
 
 From Fact \ref{isotropy.maximal}, $G$ is 
 $\SO^0(1, n+2)$ or $\SO(2, n+1)$. Again, in a  standard way, we exclude the case $G = \SO(2, n+1)$ (just because it does not contain the isotropy $\U(1) \times \SO^0(1, n)$). We have thus proved that $M = \SO^0(1, n+2)/\SO^0(1, n) \times \SO(2)$.

 \subsection{Global symmetry} \label{global.symmetry} It was proved along the investigation of cases (2) and (3) that 
 $M$ is (globally) symmetric (the global isometry with 
 $\Mink_{n+1}(\C)$ in case (2)  can be handled following the same  next argument). It remains to consider the first case, that is when the Zariski closure 
 of $H$ contains $\SU(1, n)$.
 
  Exactly as previously, by Corollary \ref{Zariski.U},  we have $\SU(1, n) \subset H$, or $\U(1) \subset H$. The last case is globally  symmetric, let us focus on the first one,  $\SU(1, n) \subset H$. 
 
  $M$ is locally isometric to a universal space $X$ of constant holomorphic sectional  curvature.
 We let the universal cover $\tilde{G}$ act on $X$.  Since the isotropy $\SU(1, n)$ of $M$ has  codimension 
 1 in the isotropy $\U(1, n)$ of $X$, $\tilde{G}$ has codimension 1 in $\Iso(X)$.  However, if $X$ is not flat, $\Iso(X)$ is a simple Lie group with no codimension 1 subgroup, since it is not locally isomorphic to $\SL_2(\R)$  since $\dim X \geq 3$ (the unique 
 simple Lie group having a codimension 1 subgroup is $\SL_2(\R)$). Therefore, $\dim G = \dim (\Iso(X))$, and in particular the isotropy of $M$ is  
 $\U(1, n)$, in particular $M$ is (globally) symmetric.
 
 Let us now consider the case of $X = \Mink_{n+1}(\C)$. Thus $\Iso(X) = \U(1, n) \ltimes \C^{n+1}$. Since $\tilde{G}$ acts (locally) transitively, it must contain some translation, that is
 $A = \tilde{G} \cap \C^{n+1} \neq 1$. The subgroup 
 $A$ is  normal in $\tilde{G}$, and is in particular $\SU(1, n)$-invariant. By irreducibility, $A = \C^{n+1}$, and thus $\tilde{G} =
 \SU(1, n) \ltimes \C^{n+1}$. The group $G$ is a quotient of $\tilde{G}$ by a discrete central subgroup. But $\tilde{G}$ has no such a subgroup. It then follows that $M = \SU(1, n) \ltimes \C^{n+1}/\SU(1, n)$, and hence $M = \Mink_{n+1}(\C)$.
  
 This finishes the proof of Theorem \ref{irreducible}. $\Box$

%\newpage

 %But it is known that $Iso(X)$ is center free.

  \section{Proof of Theorem \ref{nonproper}:  Preliminaries}
  \label{preliminaries}
  
 % \subsubsection{Notations}
  
   Let $M$ be a Hermite-Lorentz space homogeneous under the holomorphic isometric action of a semi-simple Lie group $G$ of finite center.

   %We assume the $G$-action non-proper and aim to prove, in this section,  that the isotropy %group is not amenable (unless $G$ is $\SL_2(\C)$, up to a finite cover). 

%$\bullet$ 

%We denote by $\g$ the Lie algebra of $G$.

 For $x$ in $M$, we denote by  $G_x$  its stabilizer in $G$,  $\g$ the Lie subalgebra of $G$, and   $\g_x$  the Lie subalgebra of $G_x$. The goal in this section  is to show that $\g_x$ is big; it contains nilpotent elements. 

% and $\I_x = \g_x \cap W^s_X$ (without referring to $X$ if it is fixed by the context).  

 %\medskip

%$\bullet$  Everywhere before Fact \ref{} $x$ is fixed.  

\subsection{Stable subalgebras, actions on surfaces}\label{stable}

  \subsubsection{Notations}
 
 ${}$
 
  An element $X $ in the Lie algebra $\g$  is {\bf $\R$-split}
(or  {\bf hyperbolic}) if $ad_X$ is diagonalizable with real eigenvalues. Thus $\g = \Sigma_\alpha
\g^\alpha$, 
where $\alpha$ runs over the set of eigenvalues of $ad_X$. 
Let  $$W^s_X = \Sigma_{\alpha(X) <0} \g^\alpha, \;  W^u_X = \Sigma_{\alpha(X) >0} \g^\alpha 
\; \hbox{and} \;    W^{s0}_X = \Sigma_{\alpha(X) \leq 0} \g^\alpha$$
 be respectively, the {\bf stable}, unstable and {\bf weakly-stable} sub-algebras 
of $X$.  We have in particular $\g = W^{s0}_X \oplus W^u_X$

The stable and unstable subalgebras  are nilpotent in the sense that, for $Y \in W^s_X$ (or $W^u_X$),  $ad_Y$ is a nilpotent element of $\Mat(\g)$, equivalently, $\exp ad_Y$ is a unipotent element of 
$\GL(\g)$ (this follows from relations  $[\g^\alpha, \g^\beta ] \subset \g^{\alpha + \beta}$). It then follows that  if $\h $ is an $ad_Y$-invariant subspace, then $\exp ad_Y$ determines 
a unipotent element of $\GL(\g/\h)$.

It is known that $W^s_X$ and $W^u_X$ are isomorphic; an adapted  Cartan involution sends one onto the other. In particular the codimension of $W^{s0}_X$ in $G$ equals the dimension of $W^s_X$. 
Assuming (to simplify) that $G$ is simply connected, it acts on $G/L$, where $L$ is the Lie subgroup determined by $W^{s0}_X$. Then, $\dim (G/L) = \dim W^s_X$; summarizing: 

\begin{fact} \label{action.surfaces} If for some $X$, $\dim W^s_X = 2$, then $G$ acts on a surface, that is there exists a $G$ homogeneous space of (real) dimension 2.

\end{fact}
 
%\subsubsection{Semi-simple Lie groups acting on surfaces} 

Semi-simple Lie groups satisfying the  fact can be understood:

  \begin{fact}  \label{classification.surfaces} A semi-simple Lie group $G$ acting (faithfully) on
 a  surface is  locally isomorphic to $\SL_2(\R), \SL_2(\R) \times \SL_2(\R), \SL_2(\C)$ or $\SL_3(\R)$. (it is 
well known that acting on dimension 1 implies being locally isomorphic to  $\SL_2(\R)$).
\end{fact}

\begin{proof}

 This can be derived from the classification theory of simple Lie groups. One starts observing that 
  the problem can be complexified,  that is complexified  groups act on complex surfaces; algebraically, they  possess codimension 2 complex subalgebras in their complexified algebras.  Let  the isotropy group have a Levi decomposition $S \ltimes R$. Since $S$ has a faithful 2-dimensional representation, it is locally  isomorphic to $\SL_2(\C)$. If $S^\prime $ is the affine   subgroup of $\SL_2(\C)$, then $S^\prime \ltimes R$ is solvable and has codimension 3 in $G$. Therefore, a Borel group of $G$ has codimension $\leq 3$. This implies that the cardinality of the set of positive roots is $\leq 3$ (for any associated root system). With this restriction, one observes that the (complex) rank is $\leq 2$, and consult a list of root systems to get  our mentioned groups. 
\end{proof}

\begin{example} These actions on surfaces  are  in fact   classified (up to covers). We have the projective action  of $\SL_2(\R)$ (resp. $\SL_3(\R)$) on the real projective space $\P^1(\R)$ (resp. $\P^2(\R)$). There is also the action of $\SL_2(\C)$ on the Riemann sphere $\P^1(\C)$, and the product action of $\SL_2(\R)^2$ on $\P^1(\R)^2$. Finally, the hyperbolic, de Sitter and (the punctured) affine planes are obtained as quotients of $\SL_2(\R)$ by suitable one parameter groups. It is finally possible, in some cases,  to take covers or quotients by discrete (cyclic) groups of the previous examples.  

\end{example}

  \subsection{Non-precompactness}

 \begin{fact} 
 Let $M= G/H$ be a homogeneous space where $G$ is semi-simple of finite center and acts non-properly (and faithfully)  on $M$. 
 %(and $H$ is of course a closed subgroup of $G$). 
 Then $H$ seen as the isotropy group
 of a base point, say $p$,  in not precompact in $\GL(T_pM)$. 
 
 % $H$ is not precompact in $U(1, n)$

 \end{fact}
 
 \begin{proof} By contradiction, if $H$ is precompact  then it preserves a Euclidean scalar product on $T_pM$, and hence  $G$ preserves a Riemannian metric on $G/H$ (of course $H$ is closed in $G$ since it equals the isotropy of $p$). Let us show that $H$ is compact. Indeed,  let $L$ be the isometry group of the Riemannian    homogeneous space $X = G/H$, and  $K$  its isotropy group in $L$, which is compact since the homogeneous space is of Riemannian type. Now, $H = K \cap G$.
 It is known that a semi-simple Lie group of finite center is closed in any Lie group where it lies. Therefore, $H$ is a closed subgroup of $K$, and hence compact.  
\end{proof}

\subsection{Dynamics vs Isotropy}

${}$ \\

%\begin{center}
%{\bf Dynamics versus isotropy}
%\end{center}

 For $V$ be a subspace  (in general a subalgebra) of  $\g$, its {\bf evaluation} at $x$ is the tangent subspace   $V(x) = \{ v(x)\in T_xM, v \in V\}$ (here $v$ is seen as a vector field on $M$).

\begin{fact} (Kowalsky \cite{Kow}) \label{Kow1} There exists  $X \in \g$ (depending on $x$), an   $\R$-split element,   such that $W^s_X(x)$ is isotropic.
 
\end{fact}

%\begin{proof}

%\end{proof}
In the sprit of Kowalsky's proof, we have the following precise statement.

\begin{fact} \label{Kow2} If the stabilizer algebra $\g_x$ contains a nilpotent $Y$, then any $\R$-split element $X$ of an $sl_2$-triplet $\{X, Y, Z\}$ (i.e. $[X, Y] = -Y, [X, Z] = +Z$, and  
$[Y, Z] = X$) satisfies that  $(\R X \oplus W^s_X) (x) $ is isotropic.

\end{fact}

\begin{proof} Let $L $ be  the subgroup of $G$ determined by $\{X, Y, Z\}$. It is isomorphic up to a finite cover to $\SL_2(\R)$.  A  Cartan $KAK$ decomposition yields
$\exp (tY)=   L_t \exp (s(t) X) R_t$, where $L_t$ and $R_t$ belong to the compact $\SO(2)$. 

Write $X_t =  Ad (R_t^{-1})(X)$ (for $t$ fixed), it generates the one parameter group $s \to \exp sX_t = (R_t)^{-1} \exp s X R_t $. Thus, $\exp(tY) = D_t \exp s(t) X_t$, where $D_t = L_tR_t \in \SO(2)$. 

Let $u^\alpha_t$ and $v_t^\beta$ be the eigenvectors for $ad_{X_t}$ (acting on $\g$) associated to two roots 
$\alpha$ and $\beta$.     Since $\exp tY$ preserves $\langle, \rangle$, we have
$$\langle u_t^\alpha,  v_t^\beta \rangle= \langle \exp tY u_t^\alpha, \exp tY v_t^\beta \rangle = 
e^{s(t)(\alpha + \beta)} \langle D_t u_t^\alpha, D_t v_t^\beta\rangle $$

$\bullet$ The point now is that $X_t$ converges to $X$, when $t \to \infty$. This follows from a direct computation of the KAK decomposition in $\SL_2(\R)$. It follows for the eigenvectors of $ad_X$ that 
$\langle u^\alpha, v^\beta \rangle$ is dominated by a function  of the form  $e^{s(\alpha + \beta)}$. Thus $\langle u^\alpha, u^\beta \rangle = 0$, whence $\alpha  <0$, and $ \beta \leq 0$. In particular $W^s_X(x)$ is isotropic and orthogonal to $X(x)$ (since $ad_XX = 0$).

$\bullet$ It remains to verify that $X(x)$ is isotropic. For this,  consider $M^\prime$ 
the $\SL_2(\R)$-orbit (of the base point of $M$). If the isotropy group of the $\SL_2(\R)$-action on $M^\prime$ is exactly generated by $\exp tY$, then $M^\prime$ is the affine punctured plane $\R^2 - \{0\}$. The unique $\SL_2(\R)$-invariant (degenerate) metric is 0 or   a multiple of  $d\theta^2$  
in polar coordinates $(\theta, r)$, and therefore $X(x)$  is isotropic since it coincides with $\frac{\partial}{\partial r}$. In the case where the isotropy group is bigger, the metric on $M^\prime$ must vanish (see for instance \S 2 in \cite{BFZ} for details).

% $C^t = (R_t)^{-1} \exp t R_t = \exp(t Ad R_t.(X))$, then

%$B^t = \exp tY$, $A^t = \exp tX$.
%where

\end{proof}

\begin{fact}  \label{dimension} $\g_x $ contains a nilpotent element unless  for any $X$ as 
in Fact \ref{Kow1}, $\dim W_X^s Ê\leq 2$. 
  In particular, if $G$ does not act (locally)  on  surfaces, then 
  $\g_x \cap W^s_X  \neq 0$, and $\g_x$ contains  nilpotent elements.

\end{fact}

\begin{proof} Consider the evaluation map $V \in W^s_X \to V(x) \in T_xM$. Its image is isotropic,  and thus  has at most dimension 2 ({\bf since the metric is Hermite-Lorentz}). 
 Its kernel  $ \mathfrak{I}_x =\g_x \cap W_X^s$ consists of nilpotent elements
and satisfies  $\dim (\mathfrak{I}_x) \geq \dim (W^s_X)-2$, which is positive if 
for some $X$, $\dim W^s_X > 2$, in particular if $G$ does not act (locally)  on surfaces by 
Fact \ref{action.surfaces}. 
  
   \end{proof}

 \section{Proof of Theorem \ref{nonproper}: non-amenable isotropy case}
\label{reductive}
  
 Let $(M, J, g)$ and $G$  be as in Theorem \ref{nonproper},  that is $G$ is simple, not locally isomorphic to   $\SL_2(\R), \SL_2(\C)$ or  $\SL_3(\R)$,  and acts non-properly by preserving the almost complex and  Hermite-Lorentz structures on $M$. 

Theorem \ref{nonproper} states that $M$ is exactly as in Theorem  \ref{irreducible}, that is $M$ is a global symmetric K\"ahler-Lorentz space.  It is thus natural to prove Theorem \ref{nonproper} by showing that its hypotheses imply those of Theorem \ref{irreducible}, i.e.
if the acting group is simple, and the action in non-proper, then the isotropy is irreducible.

As previously, this isotropy $H$ is a subgroup of 
$\U(1, n)$. Let us assume by contradiction that $H$ does not act irreducibly 
on $\C^{n+1}$. 

By Fact \ref{dimension}, the identity component  $H^0$ is non-precompact, which allows us using Proposition \ref{subgroups}. 

The goal of the present section is to get a contradiction assuming $H$ is non-irreducible and non-amenable. The amenable case will be treated in the next section. 

\begin{remark}
Actually, in the present  non-amenable case, all what we use from the preliminaries of \S \ref{preliminaries}, is that $H^0$ is non pre-compact (Fact \ref{dimension}). For instance if we start assuming $H$ connected, then the proof will be independent of these preliminaries and follows   from   Theorem \ref{irreducible} by the short proof below.
\end{remark}

By Proposition \ref{subgroups}, up to conjugacy, $H$ preserves $\C^{k+1}$,  for some $1 < k <n$, and
its non-compact semi-simple part is  $\SO^0(1, k)$ or $\SU(1, k)$. Let us assume here 
that it is $\SO^0(1, k)$, since the situation with $\SU(1, k)$ is even more rigid!

\subsubsection*{Integrability of distributions}
As in \S \ref{Case.SO} during the proof of Theorem \ref{irreducible}, we define a $G$-invariant  distribution $S$ on $M$, by declaring $S_p = \C^{k+1}$.

 %$\bullet$ 
  We first show that the distribution $S^\perp$ is  integrable. The obstruction to its integrability is encoded in 	 the anti--symmetric  Levi form $II: S^\perp \times S^\perp \to S$, where  $II(X, Y) $ equals the projection on $S$  of $[X, Y]$, for $X$ and $Y$   sections of $S^\perp$. 
  
   At $p$, we get a skew-symmetric form $II: \C^{n-k} \times \C^{n-k} \to \C^{k+1}$, equivariant under the actions of $\SO^0(1, k)$ on $\C^{n-k}$ and $\C^{k+1}$ respectively. Observe however that 
   $\SO^0(1, k)$ acts trivially on $\C^{n-k}$. Therefore, the image of $II$ in $\C^{k+1}$ consists of 
   fixed points, which is impossible since $\SO^0(1, k)$ has no such points (in $\C^{k+1}$).

 $\bullet$  We denote by $\mathcal S^\perp$ the so defined foliation. Before going further,  let  us notice that   
  $\SO^0(1, k) $   acts trivially on the leaf  $\mathcal S^\perp_{p}$. 
 Indeed, it preserves the induced (positive definite) Hermitian metric on $\mathcal S^\perp_{p}$. But, the derivative 
   action of $\SO^0(1, k)$ on $T_{p}\mathcal S^\perp_{p} = S^\perp_{p}$ is trivial, and hence   $\SO^0(1, k)$ acts trivially on $\mathcal S^\perp_{p}$.

$\bullet$ Let us now study $S$ itself from the point of view of integrability. We consider a similar Levi form. This time, we get an  equivariant form $\C^{k+1} \times \C^{k+1} \to
\C^{n-k}$. Since $\SU(1,k) $ acts trivially on $\C^{n-k}$, this  form is $\SO^0(1, k)$-invariant.  However, up to a constant, the K\"ahler form $\omega$ is the unique scalar $\SO^0(1, k)$-invariant form (Fact \ref{invariant.levi}). 
It follows that there  exists  $v \in \C^{n-k}$, such that $II = \omega v$. This determines a vector field $V$ on $M$ such that $V(p) = v$,  and a distribution $S^\prime = S\oplus \R V$. 

Of course, it may happen that $V = 0$, in which case $S$ is integrable. 

We claim that  $S^\prime$ is integrable. Indeed, by construction, the bracket $[X, Y]$ of two sections of $S$ belongs to $S^\prime$.  It remains to consider a bracket of the form $[V, X]$.  As previously, consideration of  an associated Levi form leads us to the following linear algebraic fact: an $\SO^0(1, k)$-invariant bilinear form
$\C^{k+1} \times \R \to \R \times \C^{n-k-1}$ must vanish.  Its proof is straightforward.

% $\bullet$ 
 
 \subsubsection*{Contradiction}  Now, we have two  foliations  $\mathcal S^\perp$  and $\mathcal S^\prime$. The group $G$ acts by preserving each of them. It also acts on   $Q$,   the (local) quotient space of $\mathcal S^\prime$, i.e. the space of its leaves. 
  However,  $\SO^0(1, k)$ acts trivially on $Q$. Indeed,  as we have seen, $\SO^0(1,k)$ acts trivially on $\mathcal S^\perp_{p}$, and this is a kind of  cross section of the quotient space  $Q$;    say,  $\mathcal S^\perp_{p}$
   meets an open set of leaves of $\mathcal S$. On this open set,  $\SO^0(1, k)$ acts trivially. By analyticity, $\SO^0(1, k)$ acts trivially on $Q$. 
   
   Thus the $G$-action on $Q$ has a non-trivial connected Kernel, and is therefore trivial 
   since $G$ is a simple Lie group. This means $Q$ is reduced to one point, that is $k= n$, which contradicts our hypothesis that $H$ is not irreducible.

\section{Proof of Theorem \ref{nonproper} in the  amenable case}
\label{non.reductive}

We continue the proof of Theorem \ref{nonproper} started in the previous section, with here
the hypothesis (by contradiction) that the isotropy $H$ is amenable. The idea of the proof is as follows. To any $x$ we associate, in a $G$-equivariant meaner, $F_x$, the {\bf asymptotic leaf} of the isotropy group $G_x$ at $x$. It is a (complex) codimension 1 lightlike geodesic hypersurface in $M$. This is got by widely general considerations (see for instance \cite{DG, Zeghib, Zeg.GAFA}).
Next, the point is to check  that $x \to F_x$ is a foliation:  $F_x \cap F_y \neq \emptyset \Longrightarrow  F_x = F_y$. Its (local) quotient space would be 
a (real) surface with a $G$-action, which is impossible by hypotheses of Theorem \ref{nonproper}; leading to that $H$ can not be amenable.

\subsubsection{Notation and Dimension}

 For $x$ in $M$, we denote by  $G_x$   its stabilizer in $G$,  $\g_x$  its Lie sub-algebra, and $\I_x = \g_x \cap W^s_X$, where $X$ is a fixed $\R$-split element as in Fact \ref{Kow2} (associated to $x$).

% (without referring to $X$ if it is fixed by the context).  

 %\medskip

%$\bullet$  Everywhere before Fact \ref{} $x$ is fixed.  

% and $G$ do not act on surfaces}

%\begin{center}

%{\bf Henceforth, in all the section, $G_x$ is supposed non-compact amenable, and till subsection \ref{}, $G$ does not act on surfaces.}

%\end{center}

%\subsubsection{The non surface-action hypothesis}

\begin{fact} 

%1) \label{dimension} $\g_x $ contains a nilpotent element, unless  for any $X$ as 
%in Fact \ref{Kow1}, $\dim W_X^s = 2$. 

%2) In particular, if $G$ does not act on a surface, 
% $\g_x \cap W^s_X  \neq 0$, in particular $\g_x$ contains a nilpotent element.  
%It then follows, up to a modification of $X$ (as in Fact \ref{Kow2}), that    $\R X (x) %\oplus W^u_X(x)$ is isotropic.

1) $X(x) \neq 0$.

2) $\dim (\I_x) \geq 2$.  
%(with the hypothesis  $G$ does not act on surfaces). 
\end{fact}

\begin{proof} 

1) By contradiction, if $X \in \g_x$, then one first proves directly (an easy case of Fact \ref{Kow1}) that also the unstable  $W^u_X$ is isotropic at $x$ and for the same reasons  $W^u_X \cap \g_x \neq 0$, say 
$Z \in W^u_X \cap \g_x$, and also $Y \in W^s_X \cap \g_x$. 

Thus $\{X, Y, Z\} \subset \g_x$.  Now, the isotropy $G_x$
embeds in the unitary group $\U(T_xM, \langle, \rangle_x)$ identified  with $\U(1, n)$. The element $\exp t X$ is an $\R$-split one parameter group that acts  on $\g_x$ with both contracted and an expanded  eigenvectors. 
 But, by the amenability hypothesis on $G_x$, it is contained in a maximal parabolic subgroup $P$ of $\U(1, n)$. However,  $P$
 % = \C^* \times \SU(1, n)\ltimes \sf{Heis}$ 
 (see \S \ref{parabolic}) has no such elements.

%  where $K$ is compact, and $\R^+$ acts on the Heisenberg group by a one parameter group of 
%contracting automorphisms.  

%acts on $\g_x$ with 

 2) Since $\I_x \neq 0$, we apply Fact \ref{Kow2} to modify $X$ if necessary and get that 
 $(\R X \oplus W^s_X)(x)$ is isotropic. Now the kernel $\I_x^\prime $ of the evaluation 
 $\R X \oplus W^s_X \to T_xM$ has dimension at least $(1 + \dim W^s_X) -2 \geq 4-2 = 2$ (since 
 we assumed that $G$ can not act on surfaces,  and hence   $\dim W^s_X \geq 3$). 
 
  It remains to check that $\I_x^\prime$ is contained in $W^s_X$ to conclude that $\I^\prime_x = \I_x$, and  obtain the desired estimation. For this assume  by contradiction that $X^\prime = X + u \in \I_x^\prime$, with  $u \in W^s_X$. It is known that  any such   $X^\prime$ is conjugate to $X$ in $\R X \oplus W^s_X$ (this is the Lie algebra of a semi-direct product of $\R$ by $\R^k$, with $\R$ acting on $\R^k$ by contraction. For $k= 1$, we get the affine group of $\R$). Therefore, we are  led to  the situation $X(x) = 0$ (for some other $x$), which we have just excluded.

\end{proof}

%\begin{center}
 \subsection{Asymptotic leaf} (see \cite{DG, Zeghib, Zeg.GAFA} for a similar situation). 
 \label{asymptotic}
 
 %\end{center}

%\begin{fact} 

%${}$ \\

Endow $M \times M$ with the metric $(+g) \oplus (-g)$. Let $f: M \to M$ be a diffeomorphism 
and $Graph(f) = \{ (x, f(x)), x \in M \}$. By definition, $f$ is isometric iff $Graph(f)$
is isotropic for $g\oplus (-g)$. Furthermore, in this case, $Graph(f)$ is a (totally) geodesic 
submanifold.   

Let $f_n$ be a  diverging  sequence in $G_x$, i.e. no sub-sequence of it converges in $G_x$.     Consider  the sequence of graphs $ Graph(f_n)$. In order to avoid global complications, let us localize things by taking     $E_n$ the connected component of $(x, x)$ in  a (small) convex neighbourhood $(O \times O) \cap Graph(f_n)$,   where $O$ is a convex neighbourhood of $x$, that is, two points of it can be joined within it by a unique geodesic.

% = \{(y, f_n(y), y \in M\} \subset M \times M$, and

Let  $V_n = Graph (D_xf_n) \subset T_xM \times T_xM$. Then, $E_n$ is the image by 
the exponential map $\exp_{(x, x)}$ of an open neighbourhood of 0 in $V_n$. 

If $V_n$ converge to $V$ in the Grassmanian space of  planes of $T_xM \times T_xM$, then $E_n$ converge in a natural way to a geodesic submanifold $E$ in $M \times M$. 
Let $V^1$ be the projection on the first factor $T_xM$. It is no longer a graph, since otherwise it would correspond to the graph of an element of $G_x$ which is a limit   of a sub-sequence of $(f_n)$ (in fact the map $f \in G_x \to Graph(D_xf)$ is a homeomorphism 
onto its image in the Grassmann space). 

Since a sequence of isometries converge iff the sequence of  inverse isometries   converge, $V$ intersects both $T_xM \times 0$ and $0 \times T_xM$ non-trivially.

Since $V_n$ is a complex (resp. isotropic) subspace, also is $V \cap (T_xM \times 0)$. Hence, because the metric on $M$ is Hermite-Lorentz, $V \cap (T_xM \times 0)$ is a complex line. Furthermore, since $V$ is isotropic, the projection $V^1$ is  a lightlike    complex hyperplane, 
with orthogonal direction $V \cap (T_xM \times 0)$.

Define similarly   $E^1$ to  be  the projection of  $E$ on $M$.  It equals the image by 
$\exp_x$ of an open subset of $V^1$. It is a lightlike  geodesic complex hypersurface (see 
\S \ref{light.geodesic}). 

%The metric on $E$ is degenerate. 

Finally, without assuming that $V_n$ converge, we consider all the  limits obtained by means of 
sub-sequences of $(f_n)$. Any so obtained  space $V^1$ (resp.   $E^1
 $) is  called {\bf asymptotic} space (resp.  leaf) of $(f_n)$ at $x$. (Observe  that different limits $V$ may have a same projection 
$V^1$). 
 
 %The asymptotic 

% One can make the following observations:

%(1) $E^1$ is   geodesic, and has tangent space $V^1$ at $x$

%(2) 

%\times T_xM$.  

%\end{fact}

\begin{fact} Let $H$ be a non-precompact amenable subgroup of  $\U(1, n)$. There is exactly one or two degenerate complex 
hyperplanes   which are   asymptotic spaces  for any sequence of $H$, and are furthermore invariant under $H$. Similarly, there exist one or two asymptotic leaves for $G_x$ (assuming it amenable and non-compact).

\end{fact}

\begin{proof} $H$ is contained in a maximal parabolic group $P$. By definition $P$ is the stabilizer of a lightlike direction $u \in \C^{1+n}$. One then observes that $(\C u)^\perp$ is a common  asymptotic space for all diverging sequences in $P$, and hence for $H$. 

Assume now that $H$ preserves two other  different degenerate complex hyperplanes  $(\C v)^\perp$ and $(\C w)^\perp$. Let us prove that $H$ is pre-compact in this case. Indeed  $H$ preserves the complex 3-space $W= \sf{Span}_\C(u, v , w)$ and 3-directions inside it,  and also $W^\perp$. Since $W^\perp$
is spacelike (the metric on it is positive), it suffices to consider the case $W^\perp = 0$. 
So the statement reduces to the compactness of the subgroup of $\U(1, 2)$ preserving 3 different $\C$-lines. This is a classical fact related to the definition of the cross ratio. 
\end{proof}

\subsubsection{Varying $x$}

${}$ \\

At our fixed $x$, we  have one or two asymptotic leaves. By homogeneity, we have the same property, one or two asymptotic leaves,  for any $y \in M$. If there are two, we arrange to choose an asymptotic leaf denoted    $F_y$, in order to insure continuity (at least) in a neighbourhood of $x$.

\begin{fact} \label{equality} Assume $y$ and $z$ near $x$. If $G_y \cap G_z $ is non-compact, then $F_y = F_z$. 

\end{fact}

\begin{proof} Let  $(f_n) $ is a diverging sequence in $G_y \cap G_z$. The fixed point set of each $f_n$ is a geodesic submanifold containing $y$ and $z$. The intersection of all 
of them when $n$ varies 
is  a geodesic submanifold $S$ containing $y$ and $z$, fixed 
by all the  $f_n$.  The graph of $f_n$
above $S$ is the diagonal of $S \times S$. Hence, $S$ is contained in the projection of any limit of $Graph(f_n)$. Therefore, any asymptotic leaf at $y$ is also
asymptotic at $z$. 
\end{proof}

%\begin{fact}  Denote by $G_x$ the stabilizer of $x$ and $\g_x$ its Lie algebra. If $G_x$ is %amenable, then there is a unique  $F_x$ a codimension 1... The isotropic leaf at $x$ is a  %surface $F^\perp_x$ whose tangent space at $x$ is a $\C$-line $\N_x$.

%\end{fact}

\subsection{Geometry}
%\begin{center}

%\end{center}

${}$

%Fix $x$ in $M$.

% and assume that $G$ has no quotient $G/H$ of (real) dimension $\leq 2$.

% and choose $X$ as in Fact \ref{Kow2}.

\begin{fact} (Remember  the notation $ \mathfrak{I}_x =\g_x \cap W_X^s$). If 
$X$ can be chosen such that $\dim  \I_x \geq 3$, then for any $y \in F_x$, $\g_y \cap \I_x \neq 0$. In particular $F_y = F_x$.
\end{fact}

\begin{proof} By Corollary \ref{preservation}, the group generated by $\I_x$ preserves each leaf of the characteristic foliation $\mathcal N$ of $F_x$.  Equivalently, the evaluation $\I_x(y) $ is contained in the tangent space of the characteristic leaf 
%\subset 
$\mathcal N_y$, for any $y \in F_x$. Since  $\dim \I_x \geq 3$, and $\dim {\mathcal N}_y = 2$, we conclude that $\I_x \cap \g_y \neq 0$, and hence $F_x = F_y$ by Fact \ref{equality}.

%Set $L= F_x$, and for $y \in L$, let $h_y$ the restriction of the metric $g_x$ 
%on $T_yL$. It is degenerate with  Kernel a $\C$-line $N_y$. The metric on $T_yL /N_y$ is %positive. 

%.It   is a degenerate complex hypersurface. The restriction o

\end{proof}

%\begin{fact}
%The algebra $W^u_X$ preserves each light leaf $Z$ of $F_x$. That is, $W^u_X$ is everywhere %isotropic and tangent to $F_x$, for any $y \in F_x$.
%\end{fact}

\subsubsection{Same conclusion in the other case}

%${}$ \\

%under the hypothesis that $G$ can not act on surfaces}

%\subsubsection*{Lower dimension case} 

% and assume that $G$ has no quotient $G/H$ of (real) dimension $\leq 2$.

Assume now that for any choose of $X$,  $\dim \I_x = 2$, say $\I_x = \sf{Span}\{a, b\}$, and let $A = \exp a$ and
$B = \exp b$. The set of fixed points of a nilpotent  element, say $a$ (or equivalently a unipotent element $A$)  is a geodesic submanifold $Fix(a)$ of complex codimension 2 (one cheeks this for elements of $\U(1, n)$). 

Choose   $y \in Fix(a)$, then by Fact \ref{equality}, $F_y = F_x$. Since $a \in \g_y$, we can apply  Facts \ref{Kow2} and \ref{dimension} and get for the same $X$ that $\I_y =
 \g_y \cap W^s_X$ has dimension $ \geq 2$.
 
 %, in fact by homogeneity of $M$, $y $
 %plays the same role as $x$, and hence $\dim \I_y = 2$. 

 Let $\I$ be the subalgebra of $W^s_X$ generated by $\I_x$ and $\I_y$. It is nilpotent since contained in $W^s_X$. 
 Assume furthermore that $y \in Fix(a) - Fix(b)$, then 
  $\I_x \neq \I_y$, and hence $\dim \I \geq 3$. 
  % since by 
 %hypothesis, $y \notin Fix(b)$. 
%  Select $c$ a second generator of $\I_y$ besides $a$.   
% Let $\I $ be the subalgebra of $W^s_X$ generated by $\sf{Span}\{a, b, c\}$,  then $\dim \I %\geq 3$, and it is nilpotent (since contained in $W^s_X$).  
 
 Since it is generated by $\I_x$ and $\I_y$, $\I$ preserves individually the leaves of the characteristic foliation $\mathcal N$ on $F_x$. As above, by the inequality on dimensions, 
 any $z \in F_x$ is fixed by a non-trivial element of $\I$. Therefore, by Fact \ref{equality}, $F_z = F_x$.

  %In fact, by homogeneity of $M$, as for $x$, we have $\dim \I_y = 2$.

%\begin{proof}

%\end{proof}

%\begin{fact} For any $y \in F_x$, $\h_x \cap \h_y \neq 0$, and thus   $F_y = F_x$.
%\end{fact}

\subsection{End, Contradiction} The previous conclusion means that  two asymptotic leaves    are disjoint or equal,  that is they define a foliation of $M$, of (real) codimension 2.  This foliation is $G$ invariant. Therefore, $G$ acts on the (local) quotient space of the foliation. This contradicts our hypothesis that  $G$ does not act (locally) on surfaces.

%The map  $y \to F_y$ defines a foliation of $M$.

 %\cite{DMZ}


\begin{thebibliography}{10}
  
\bibitem{A1}  S. Adams, Orbit nonproper actions on Lorentz manifolds, GAFA, Geom. Funct. Anal. 11:2 (2001), 201-243.



\bibitem{A2}S. Adams, Dynamics of semisimple Lie groups on Lorentz manifolds, Geom. Ded. 105 (2004), 1-12.

\bibitem{ADZ} A. Arouche, M. Deffaf and A. Zeghib, 
On Lorentz dynamics: From group actions to warped products via homogeneous spaces.   
Trans. Amer. Math. Soc. 359  (2007), 1253-12 


\bibitem{Romero1}  M. Barros and A.  Romero, 
Indefinite K\"ahler manifolds. 
Math. Ann. 261 (1982), no. 1, 55-62.

\bibitem{BFZ} E. Bekkara, C. Frances and A. Zeghib, Actions of semisimple Lie groups preserving a degenerate Riemannian metric. 
Trans. Amer. Math. Soc, 362 (2010), 2415 - 2434.

\bibitem{Berger} M. Berger, Les espaces sym\'etriques non compacts. 
%Les espaces sym\'etriques non compacts, 
%Ann. Sci. \'ecole Norm. 
Ann. Sci. \'Ecole Norm. Sup.  74 (1957) 85 - 177.



\bibitem{Besse} A. Besse,  Einstein manifolds. Reprint of the 1987 edition. Classics in Mathematics. Springer-Verlag, Berlin, 2008.


\bibitem{BoZ} C. Boubel and A. Zeghib, Isometric actions of Lie subgroups of the Moebius group, Nonlinearity 17 (2004), 1677-1688.


\bibitem{Cohen} M. Cahen, J. Leroy, M. Parker, F. Tricerri and L. Vanhecke, Lorentz manifolds modeled on a Lorentz symmetric space, J. Geom. Phys. 7 (1991), 571-581.

\bibitem{D'A} G. DÕAmbra, Isometry groups of Lorentz manifolds. Invent. Math. 92 (1988), no. 3, 555-565.




\bibitem{DG} G. DÕAmbra and  M. Gromov, Lectures on transformation groups: geometry and dynamics, in ÒSurveys in Differential Geometry (Cambridge, MA, 1990)Ó, Lehigh Univ., Bethlehem, PA (1991), 19-111


\bibitem{DMZ} M. Deffaf, K. Melnick and A. Zeghib
Actions of noncompact semisimple groups on Lorentz manifolds, 
Geom. Funct. Anal.   18 (2008) 463-488

\bibitem{D-L}  A. Di Scala and  T. Leistner,   Connected subgroups of $\SO(2,n)$ acting irreducibly on $\R^{2,n}$. Israel J. Math. 182 (2011), 103-121.


 \bibitem{DO} A. Di Scala   and C. Olmos,  The geometry of homogeneous submanifolds of hyperbolic space,  Math. Z.
237 199-209.


\bibitem{Ga} A. Galaev, Classification of connected holonomy groups of pseudo-K\"ahlerian manifolds of index 2,   arXiv:math/0405098. 

 \bibitem{GL}  A. Galaev 	and T. Leistner, Recent developments in pseudo-Riemannian holonomy theory. Handbook of pseudo-Riemannian geometry and supersymmetry, 581Ð627, IRMA Lect. Math. Theor. Phys., 16, Eur. Math. Soc., Z\"urich, 2010. 
 
 \bibitem{Goldman}
 
 W. Goldman,  Complex hyperbolic geometry. Oxford Mathematical Monographs. Oxford Science Publications. The Clarendon Press, Oxford University Press, New York, 1999. 
 
 \bibitem{Kar} 
 F. Karpelevic, Surfaces of transitivity of a semisimple subgroup of the group
of motions of a symmetric space. Doklady Akad. Nauk SSSR (N.S.), 93:401-404,
1953.
 
 
\bibitem{Kath}
I. Kath, I, M.  Olbrich,  On the structure of pseudo-Riemannian symmetric spaces. Transform.
Groups 14 (2009), no. 4, 847- 885.


\bibitem{Kob-Nom} S. Kobayashi and K.  Nomizu,   Foundations of differential geometry. Vol. II. Reprint of the 1969 original. Wiley Classics Library. A Wiley-Interscience Publication, 
New York, 1996. 
 
 \bibitem{Kow} N. Kowalsky, Noncompact simple automorphism groups of Lorentz manifolds, Ann. Math. 144 (1997), 611-640.
 
 \bibitem{Mos}
 
 G. Mostow, Some new decomposition theorems for semi-simple groups,  Mem. Amer. Math. Soc. 1955, (1955)


%\bibitem{Romero2}  A. Romero and Y.   Suh, Differential geometry of indefinite complex submanifolds in indefinite complex space forms,  Extracta %Math. 19 (2004), no. 3, 339-398.



\bibitem{Zeghib} A. Zeghib, 
Remarks on Lorentz symmetric spaces,   
Compositio Math. 140 (2004) 1675-1678


\bibitem{Zeg.GAFA} A. Zeghib, Isometry groups and geodesic foliations of Lorentz manifolds, Part I: Foundations
of Lorentz dynamics, Geom. Funct. Anal. Vol. 9 (1999), 775-822.

\end{thebibliography}
\end{document}